\documentclass{amsart}
\usepackage{fixltx2e}
\newcounter{cases}
\newcounter{subcases}[cases]

\usepackage{amsmath, amsthm, amscd, amsfonts, amssymb, graphicx, color,float,pgf,tikz}
\usepackage{amssymb,fontenc}
\usepackage{latexsym,wasysym,mathrsfs}
\usepackage{hyperref}
\usepackage[utf8]{inputenc}
\usepackage[T1]{fontenc}
\usepackage{booktabs}
\usepackage{geometry}
\usepackage{subcaption}
\usepackage{soul}
\newcommand{\str}[2]{\left[ \genfrac{}{}{0pt}{}{#1}{#2} \right]_r}
\newcommand{\cyc}{\mathcal{C}}
\newtheorem{theorem}{Theorem}[section]
\newtheorem{lemma}[theorem]{Lemma}
\newtheorem{proposition}[theorem]{Proposition}
\newtheorem{corollary}[theorem]{Corollary}
\newtheorem{conjecture}[theorem]{Conjecture}
\newtheorem{problem}[theorem]{Problem}
\theoremstyle{definition}
\newtheorem{definition}[theorem]{Definition}
\newtheorem{example}[theorem]{Example}

\newtheorem{remark}[theorem]{Remark}

\newcommand{\stirlgraph}[2]{\genfrac{[}{]}{0pt}{}{#1}{#2}}

\def\st#1#2{\begin{bmatrix} #1 \\ #2 \end{bmatrix}}

\author{Daniel Yaqubi$^{*}$}
\address{\bf $^*$Department of Computer science, University of Torbat-e Jam, Torbat-e Jam, Iran.}
\email{yaqubi@tjamcaas.ac.ir, or daniel\_yaqubi@yahoo.es}
\author{Madjid Mirzavaziri}
\address{\bf Department of Pure Mathematics, Faculty of Mathematical Sciences,  Ferdowsi University of Mashhad, P. O. Box 1159-91775, Mashhad, Iran.}
\email{mirzavaziri@gmail.com}


\title{On the Graphical $r$-Stirling Numbers of the First Kind for Specific Graph Families}

\date{\today}
\begin{document}
\maketitle

\begin{abstract}
This paper investigates the \textbf{graphical $r$-Stirling numbers of the first kind}, denoted by $\str{G}{k}$, which enumerate partitions of a vertex set $V(G)$ into $k$ disjoint cycles such that $r$ specified vertices occupy distinct blocks. We establish closed-form expressions and recursive identities for fundamental graph families, including \textbf{Path} ($P_n$), \textbf{Cycle} ($C_n$), \textbf{Star} ($S_n$), \textbf{Wheel} ($W_n$), and \textbf{Fan} ($F_n$) graphs. 

A primary focus of this study is the \textbf{statistical characterization} of the cycle distribution. We derive explicit formulas for the \textbf{mean} and \textbf{variance} of these numbers, extracted from the structural properties of the $r$-cycle polynomials. These results provide a rigorous measure of the average cycle density and variability across different graph topologies, bridging the gap between algebraic combinatorics and the structural analysis of restricted permutations.

\end{abstract}
\section{Introduction}

The Stirling numbers of the first kind ${n \brack k}$ are fundamental objects in
enumerative combinatorics, counting permutations of an $n$-element set with
exactly $k$ disjoint cycles. These numbers play a central role in the study of
permutation statistics, algebraic combinatorics, and probability theory
\cite{Goncharov1944,Feller1968}. Classical results of Goncharov and Feller
established deep connections between Stirling numbers, generating polynomials,
and asymptotic normality, leading to a Central Limit Theorem for the number of
cycles in random permutations.

An important refinement was introduced by Broder \cite{Broder1984}, who defined
the \emph{$r$-Stirling numbers of the first kind} ${n \brack k}_r$ by imposing the
restriction that the first $r$ elements must lie in distinct cycles. These
restricted Stirling numbers interpolate between classical permutations and more
constrained cycle structures, and they have been studied extensively in
connection with restricted permutations, algebraic identities, and probabilistic
limit laws.

Motivated by the interaction between graph structure and cycle decompositions,
Barghi and DeFord \cite{Barghi2023} introduced the \emph{graphical Stirling
numbers of the first kind} ${G \brack k}$ for an arbitrary graph $G$. These
numbers enumerate partitions of the vertex set $V(G)$ into $k$ disjoint cycles
supported by the edges of $G$, allowing both $1$-cycles and $2$-cycles. When
$G$ is the complete graph $K_n$, the numbers ${G \brack k}$ reduce to the
classical Stirling numbers of the first kind. The sum $\sum_k {G \brack k}$,
called the \emph{graphical factorial} of $G$, counts all cycle covers of $G$.
Although cycle cover enumeration is computationally difficult in general,
explicit formulas and structural results can be obtained for sparse or highly
structured graph families.

In this paper, we introduce and study the \emph{graphical $r$-Stirling numbers of
the first kind}, denoted ${G \brack k}_r$. These numbers count partitions of
$V(G)$ into $k$ disjoint cycle blocks subject to the additional constraint that
$r$ specified vertices lie in distinct cycles. This definition simultaneously
generalizes Broder’s $r$-Stirling numbers and the graphical Stirling numbers of
the first kind, providing a unified framework for studying restricted cycle
decompositions in graphs and modeling combinatorial situations in which
designated vertices must participate in closed paths.

Our first objective is to extend enumerative results for ${G \brack k}$ to this
$r$-restricted setting. We derive explicit formulas, recurrence relations, and
$r$-cycle polynomials for several fundamental families of graphs, including
paths ($P_n$), cycles ($C_n$), stars ($S_n$), wheels ($W_n$), and fan graphs
($F_n$). 

Beyond exact enumeration, we investigate the \emph{algebraic properties} of
graphical $r$-Stirling numbers. Building on Harper’s seminal work on the
real-rootedness and log-concavity of Stirling polynomials \cite{Harper1967}, as
well as subsequent developments by Liu, Yang, and Zhang
\cite{LiuYangZhang2012,LiuYangZhang2014}, we study unimodality, log-concavity, and
real-rootedness of the associated $r$-cycle polynomials for various graph
families. These algebraic properties provide strong structural insight and play
a crucial role in probabilistic analysis.

From a probabilistic perspective, classical results of Goncharov
\cite{Goncharov1944} and Feller \cite{Feller1968} established a Central Limit
Theorem for the number of cycles in permutations. More recently, asymptotic
normality and related limit laws have been studied for generalized Stirling
numbers and graph-based combinatorial statistics
\cite{DuncanPeele2009,GalvinThanh2013}. In this work, we extend these ideas to the
graphical $r$-Stirling setting. Using analytic properties of generating
polynomials and moment methods, we derive explicit formulas for the mean and
variance of the number of cycles and establish Gaussian limiting behavior for
several families of graphs as the number of vertices tends to infinity.

Finally, we examine asymptotic behavior across growing families of graphs and
analyze how sparsity, degree distribution, and structural constraints influence
both enumerative and probabilistic properties. The paper concludes with several
conjectures and open problems concerning asymptotic normality, real-rootedness,
and algebraic behavior of graphical $r$-Stirling polynomials, suggesting
promising directions for future research in graph enumeration, algebraic
combinatorics, and probabilistic graph theory.
\section{Preliminaries and notation}

%

Let $G=(V,E)$ be a simple graph with vertex set $V=[n]$ and edge set $E$. A \textit{graphical cycle partition} of $G$ into $k$ blocks is a partition of the vertex set $\pi(G) = \{C_1, C_2, \dots, C_k\}$ such that for each block $C_i$, one of the following conditions is satisfied:
\begin{itemize}
	\item[(i)] $|C_i| = 1$ (a $1$-cycle) or $|C_i| = 2$ (a $2$-cycle).
	\item[(ii)] If $|C_i| > 2$, the induced subgraph $G|_{C_i}$ is Hamiltonian; that is, it is connected and contains at least one cycle that visits each vertex in $C_i$ exactly once.
\end{itemize}

In the case of $|C_i| > 2$, we treat $C_i$ as a specific oriented Hamiltonian cycle within the induced subgraph. The \textit{graphical Stirling number of the first kind}, denoted by ${G \brack k}$, represents the total number of such partitions of $V$ into exactly $k$ cycles. This definition generalizes the classical Stirling numbers by restricting the cycle formation to the edge set of $G$.
\begin{definition}
Let $G$ be a graph of order $n$. The \emph{graphical cycle polynomial} $\mathcal{C}(G, x)$ is the generating function for the graphical cycle partition numbers ${G \brack k}$:
\begin{equation*}
    \mathcal{C}(G, x) = \sum_{k=1}^{n} {G \brack k} x^k
\end{equation*}
\end{definition}
\begin{remark}
Although both the matching polynomial and the graphical cycle polynomial of the
first kind are  satisfy linear recurrences on some graphs,
they encode fundamentally different combinatorial structures.
\end{remark}
\begin{example}
Let $W_4$ be the wheel graph of order $n=5$ defined by the join $K_1 \vee C_4$, with vertex set $V(W_4) = \{v_0, 1, 2, 3, 4\}$. Here, $v_0$ represents the central hub, and the induced subgraph of $\{1, 2, 3, 4\}$ is the cycle $C_4$. The graph contains $|E(W_4)| = 8$ edges and 4 triangles ($C_3$).

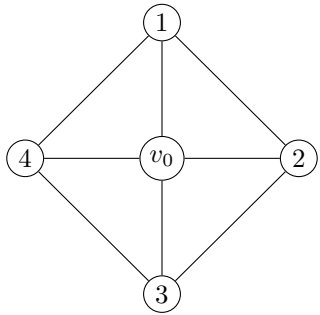
\begin{figure}[h]
    \centering
    \begin{tikzpicture}[scale=1.2, every node/.style={circle, draw, inner sep=2pt}]
        \node (v0) at (0,0) {$v_0$};
        \node (1) at (0,1.5) {$1$};
        \node (2) at (1.5,0) {$2$};
        \node (3) at (0,-1.5) {$3$};
        \node (4) at (-1.5,0) {$4$};
        \draw (1) -- (2) -- (3) -- (4) -- (1);
        \draw (v0) -- (1); \draw (v0) -- (2); \draw (v0) -- (3); \draw (v0) -- (4);
    \end{tikzpicture}
    \caption{The wheel graph $W_4$ with hub $v_0$ and perimeter cycle $C_4$.}
    \label{fig:wheel4}
\end{figure}

The graphical Stirling numbers ${W_4 \brack k}$ enumerate partitions of $V(W_4)$ into $k$ disjoint cycles supported by the graph structure. The distribution is summarized in Table \ref{tab:wheel_counts}.

\begin{table}[h]\label{W4}
\centering
\caption{Cycle partition counts for the Wheel graph $W_4$.}
\label{tab:wheel_counts}
\begin{tabular}{|c|l|l|c|}
\hline
\textbf{$k$} & \textbf{Partition Type} & \textbf{Combinatorial Derivation} & \textbf{${W_4 \brack k}$} \\ \hline
5 & $(C_1, C_1, C_1, C_1, C_1)$ & Unique partition into singletons & 1 \\ \hline
4 & $(C_2, C_1, C_1, C_1)$ & Total number of edges: $|E(W_4)|$ & 8 \\ \hline
3 & $(C_3, C_1, C_1)$ & 4 triangles containing $v_0$ & 4 \\
  & $(C_2, C_2, C_1)$ & Disjoint edge pairs (2 perimeter-only, 8 hub-spoke) & 10 \\ \cline{3-4} 
  & & \textbf{Total for $k=3$} & \textbf{14} \\ \hline
2 & $(C_4, C_1)$ & 1 perimeter cycle + 4 hub-inclusive 4-cycles & 5 \\
  & $(C_3, C_2)$ & Disjoint triangle-edge pairs (e.g., $v_0 1 2 / 3 4$) & 4 \\ \cline{3-4} 
  & & \textbf{Total for $k=2$} & \textbf{9} \\ \hline
1 & $(C_5)$ & 4 Hamiltonian subgraphs $\times$ 2 orientations each & 8 \\ \hline
\end{tabular}
\end{table}

The resulting graphical cycle partition vector for $W_4$ is $(8, 9, 14, 8, 1)$, defining the cycle polynomial:
\begin{equation*}
\mathcal{C}(W_4, x) = x^5 + 8x^4 + 14x^3 + 9x^2 + 8x
\end{equation*}
\end{example}
The basic properties and specific formulas for well-known graph families are summarized in Table~\ref{tab:values}.
\begin{table}[h]
\centering
\renewcommand{\arraystretch}{1.5}
\begin{tabular}{ll}
\hline
\textit{Property / Graph Class} & \textit{Formula for ${G \brack k}$} \\ \hline
General Identity (Order $n$) & ${G \brack n} = 1$ \\
General Identity (Edges $E$) & ${G \brack n-1} = |E|$ \\
Hamiltonicity & ${G \brack 1} = 2 \times (\text{No. of Hamiltonian cycles})$ \\
Empty Graph $E_n$ & ${E_n \brack n} = 1$; else $0$ \\
Complete Graph $K_n$ & ${K_n \brack k} = {n \brack k}$ \\
Path Graph $P_n$ & ${P_n \brack k} = \binom{k}{n-k}$ \\
Cycle Graph $C_n$ ($k>1$) & ${C_n \brack k} = \binom{k}{n-k} + \binom{k-1}{n-k-1}$ \\
Cycle Graph $C_n$ ($k=1$) & ${C_n \brack 1} = 2$ \\
Complete Bipartite $K_{n,m}$ & $\sum_{i=0}^{\min(n,m)} \binom{m}{i}\binom{n}{i} {i \brack k+2i-(m+n)} i!$ \\ \hline
\end{tabular}
\caption{Properties and explicit formulas for graphical Stirling numbers of the first kind.}
\label{tab:values}
\end{table}

\section{General Statements and Graph Operations}

In this section, we formalize the graph-theoretic operations required to establish our recurrence relations. Let $G=(V,E)$ be a simple graph of order $n$.

\begin{itemize}
    \item \textbf{Vertex Deletion}: For a vertex $v_i \in V$, let $G - v_i$ denote the subgraph induced by $V \setminus \{v_i\}$. More generally, for a subset $U \subset V$, $G - U$ denotes the subgraph induced by $V \setminus U$.
    
    \item \textbf{Vertex Attachment (Gluing)}: Let $G +_{v_1} u$ denote the graph obtained by adding a vertex $u \notin V$ and the edge $e = (v_1, u)$, where $v_1 \in V$ is the point of attachment.
    
    \item \textbf{Brooming Operation}: Let $U = \{u_1, u_2, \dots, u_m\}$ be a set of $m$ vertices such that $V \cap U = \emptyset$. We define the \textit{brooming} of $G$ at $v_1 \in V$, denoted by $G +_{v_1} U$, as the graph with vertex set $V \cup U$ and edge set $E \cup \{(v_1, u_i) : 1 \leq i \leq m\}$.
\end{itemize}

\begin{example}
Let $G \cong C_3$ be the cycle graph with vertices $\{v_1, v_2, v_3\}$. Figure~\ref{fig:ops} illustrates the transformation of $C_3$ under deletion, attachment, and brooming. Original vertices are shaded in blue, while added vertices and edges are highlighted in red.

\begin{figure}[h]
    \centering
    \begin{tikzpicture}[scale=0.9, 
        orig/.style={circle, draw, fill=blue!15, inner sep=2.5pt, font=\small},
        newv/.style={circle, draw, fill=red!15, inner sep=2.5pt, font=\small}]
        
        \begin{scope}[shift={(0,0)}]
            \node[orig] (v1) at (90:1) {$v_1$};
            \node[orig] (v2) at (210:1) {$v_2$};
            \node[orig] (v3) at (330:1) {$v_3$};
            \draw[thick] (v1) -- (v2) -- (v3) -- (v1);
            \node[draw=none, below=1.2cm] at (0,0) {(a) $G \cong C_3$};
        \end{scope}

        \begin{scope}[shift={(3.5,0)}]
            \node[orig] (v1) at (90:1) {$v_1$};
            \node[orig] (v2) at (210:1) {$v_2$};
            \draw[thick] (v1) -- (v2);
            \node[draw=none, below=1.2cm] at (0,0) {(b) $G - v_3$};
        \end{scope}

        \begin{scope}[shift={(7,0)}]
            \node[orig] (v1) at (90:1) {$v_1$};
            \node[orig] (v2) at (210:1) {$v_2$};
            \node[orig] (v3) at (330:1) {$v_3$};
            \node[newv] (u) at (90:2.2) {$u$};
            \draw[thick] (v1) -- (v2) -- (v3) -- (v1);
            \draw[thick, red, dashed] (v1) -- (u);
            \node[draw=none, below=1.2cm] at (0,0) {(c) $G +_{v_1} u$};
        \end{scope}

        \begin{scope}[shift={(10.5,0)}]
            \node[orig] (v1) at (90:1) {$v_1$};
            \node[orig] (v2) at (210:1) {$v_2$};
            \node[orig] (v3) at (330:1) {$v_3$};
            \node[newv] (u1) at (60:2.2) {$u_1$};
            \node[newv] (u2) at (90:2.2) {$u_2$};
            \node[newv] (u3) at (120:2.2) {$u_3$};
            \draw[thick] (v1) -- (v2) -- (v3) -- (v1);
            \draw[thick, red, dashed] (v1) -- (u1); 
            \draw[thick, red, dashed] (v1) -- (u2); 
            \draw[thick, red, dashed] (v1) -- (u3);
            \node[draw=none, below=1.2cm] at (0,0) {(d) $G +_{v_1} U$};
        \end{scope}
    \end{tikzpicture}
    \caption{Visual representation of graph operations at vertex $v_1$.}
    \label{fig:ops}
\end{figure}
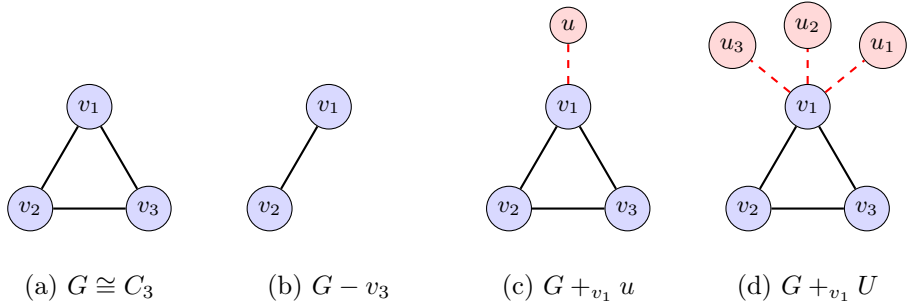
\end{example}

\begin{lemma}\label{lem:properties}
Let $G$ be a simple graph. The graphical Stirling numbers ${G \brack k}$ satisfy the following recursive and structural identities:
\begin{enumerate}
    \item \textbf{Isolation Shift:} Let $I \subseteq V(G)$ be a set of $m$ isolated vertices. For $k \geq m$:
    \begin{equation*}
        {G \brack k} = {G - I \brack k - m}
    \end{equation*}
    
    \item \textbf{Disjoint Union Convolution:} Let $G$ be the disjoint union of two graphs $G_1$ and $G_2$. Then:
    \begin{equation*}
        {G_1 \cup G_2 \brack k} = \sum_{j=1}^{k-1} {G_1 \brack j} {G_2 \brack k - j}
    \end{equation*}
      \item
   The graphical cycle polynomial of their union is given by the product of their individual polynomials:
\begin{equation*}
    \mathcal{C}(G_1 \cup G_2, x) = \mathcal{C}(G_1, x) \cdot \mathcal{C}(G_2, x)
\end{equation*}
\end{enumerate}
\end{lemma}

\begin{proof}
(1) Let $I = \{v \in V(G) : \deg(v) = 0\}$. By definition, every block in a graphical cycle partition must induce a cycle $C_{\ell}$. Since vertices in $I$ have no incident edges, they cannot participate in induced cycles of length $\ell \geq 2$. Consequently, each vertex $v \in I$ must form a trivial 1-cycle ($C_1$). Since these $m$ isolated vertices are necessarily partitioned into $m$ distinct cycles, the problem reduces to partitioning the remaining vertex set $V(G) \setminus I$ into $k - m$ induced cycles. This yields the identity ${G - I \brack k - m}$.

(2) Since $G_1$ and $G_2$ are disjoint, any subset of $V(G_1 \cup G_2)$ that induces a cycle must be a subset of either $V(G_1)$ or $V(G_2)$. Thus, a partition of the union into $k$ cycles is uniquely determined by the union of a $j$-cycle partition of $G_1$ and a $(k-j)$-cycle partition of $G_2$. Summing over all valid $j$ yields the discrete convolution.

(3)  By the definition of the graphical cycle polynomial and the convolution property of the cycle partition numbers for disjoint unions:
\begin{align*}
    \mathcal{C}(G_1 \cup G_2, x) &= \sum_{k} {G_1 \cup G_2 \brack k} x^k \\
    &= \sum_{k} \left( \sum_{j} {G_1 \brack j} {G_2 \brack k-j} \right) x^k \\
    &= \left( \sum_{j} {G_1 \brack j} x^j \right) \left( \sum_{m} {G_2 \brack m} x^m \right) \\
    &= \mathcal{C}(G_1, x) \cdot \mathcal{C}(G_2, x)
\end{align*}
This completes the proof.
\end{proof}

\begin{corollary}\label{thm:chu-vander}
Let $G$ be the disjoint union of components $G_1, G_2, \dots, G_m$. The graphical Stirling numbers satisfy the multinomial convolution:
\begin{equation*}
    {G \brack k} = \sum_{j_1 + \dots + j_m = k} \left( \prod_{i=1}^m {G_i \brack j_i} \right)
\end{equation*}
Equivalently, the graphical cycle polynomial $\mathcal{C}(G, x) = \sum_k {G \brack k} x^k$ is multiplicative over disjoint components:
\begin{equation*}
    \mathcal{C}\left( \bigcup_{i=1}^m G_i, x \right) = \prod_{i=1}^m \mathcal{C}(G_i, x)
\end{equation*}
\end{corollary}

%

\begin{lemma}\label{spike}
Let $G$ be a simple graph and $v \notin V(G)$. For any $w \in V(G)$, let $G+_w v$ be the graph formed by attaching $v$ to $w$ via a pendant edge. The graphical cycle partition numbers satisfy:
\begin{equation*}
    {G+_w v \brack k} = {G \brack k-1} + {G - w \brack k-1}
\end{equation*}
\end{lemma}

\begin{proof}
Consider the vertex $v$ in $G+_w v$. Since $v$ is a pendant vertex with $N(v) = \{w\}$, it can only induce a cycle of length 1 or 2.
\begin{itemize}
    \item If $v$ induces a $C_1$, the remaining vertices $V(G)$ must be partitioned into $k-1$ cycles. There are ${G \brack k-1}$ such ways.
    \item If $v$ induces a $C_2$, it must be the edge $vw$. The remaining vertices $V(G) \setminus \{w\}$ must be partitioned into $k-1$ cycles. This corresponds to the graph $G-w$, yielding ${G - w \brack k-1}$ ways.
\end{itemize}
Summing these disjoint cases gives the result.
\end{proof}
Now, we derive the graphical cycle polynomial for the cycle graph $C_n$ and $P_n$. This derivation utilizes Kaplansky’s First Lemma \cite{God}, which provides the combinatorial count for matchings in circular arrangements.

\begin{lemma}[Kaplansky's Lemma for Circular Arrangements]\label{Kap}
The number of ways to select $m$ non-adjacent edges from a cycle graph $C_n$ (equivalent to selecting $m$ non-consecutive objects from $n$ objects arranged in a circular fashion) is given by:
\begin{equation*}
     \frac{n}{n-m} \binom{n-m}{m}
\end{equation*}
where $0 \leq m \leq \lfloor n/2 \rfloor$.
\end{lemma}

\begin{theorem}\label{Pn}
The graphical cycle polynomial of the path graph $P_n$ is given by:
\begin{equation*}
    \mathcal{C}(P_n, x) = \sum_{k=\lceil n/2 \rceil}^{n} \binom{k}{n-k} x^k
\end{equation*}
where the coefficients $\binom{k}{n-k}$ correspond to the number of matchings of size $n-k$ in $P_n$.
\end{theorem}

\begin{proof}
In the path graph $P_n$, we have only 1-cycles and 2-cycles.  Thus, any graphical cycle partition of $V(P_n)$ into $k$ subsets must consist solely of $m$ induced $C_2$ subgraphs (edges) and $s$ induced $C_1$ subgraphs (vertices). 

The total number of vertices is $2m + s = n$ and the total number of cycles is $m + s = k$, which implies $m = n-k$. According to Kaplansky's Lemma \ref{Kap}, the number of ways to choose $j$ items from $m$ items in a row such that no two are consecutive is given by:
\begin{equation*}
    \st{P_n}{k} = \binom{(n-1) - (n-k) + 1}{n-k} = \binom{k}{n-k}
\end{equation*}
Summing over all possible values of $k$ (where $n-k \leq k$ implies $k \geq n/2$) yields the result.
\end{proof}


\begin{theorem}\label{Cn}
The graphical cycle polynomial of the cycle graph $C_n$ for $n \geq 3$ is given by:
\begin{equation*}
    \mathcal{C}(C_n, x) = 2x + \sum_{k=\lceil n/2 \rceil}^{n} \stirlgraph{C_n}{k} x^k
\end{equation*}
where the coefficients $\stirlgraph{C_n}{k} = \frac{n}{k} \binom{k}{n-k}$ for $k \geq 2$ correspond to the number of matchings of size $n-k$ in $C_n$.
\end{theorem}

\begin{proof}
By definition, the graphical cycle polynomial $\mathcal{C}(G, x)$ enumerates the partitions of the vertex set $V(G)$ into $k$ disjoint cycles, where a partition into $k$ cycles contributes to the coefficient of $x^k$. For the cycle graph $G = C_n$, the admissible cycle lengths in a vertex partition are strictly constrained by the graph's topology.

\paragraph{I. The Case $k=1$ (Hamiltonian Cycle):}
For $k=1$, the vertex set must be covered by a single cycle of length $n$. While $C_n$ contains exactly one such subgraph, we account for the two possible orientations (clockwise and counter-clockwise) of the Hamiltonian cycle, consistent with the treatment of cycles in directed contexts (analogous to Stirling numbers of the first kind). This yields the term:
\begin{equation*}
    2x^1 = 2x.
\end{equation*}

\paragraph{II. The Case $k \geq 2$.}
For $k \geq 2$, the partition cannot include the Hamiltonian cycle. Furthermore, for $n \geq 3$, the cycle graph $C_n$ does not contain any vertex-disjoint cycles of length $3 \leq \ell < n$. Consequently, any partition of $V(C_n)$ into $k$ components must consist exclusively of:
\begin{itemize}
    \item $m$ cycles of length 2 (representing a matching of $m$ disjoint edges),
    \item $s$ cycles of length 1 (representing isolated vertices).
\end{itemize}

To determine the number of components $k$, we consider the following system:
\begin{align}
    2m + s &= n \quad \text{(Vertex constraint)} \\
    m + s &= k \quad \text{(Component constraint)}
\end{align}
Solving this system yields $m = n - k$. Thus, a vertex partition into $k$ components is bijectively related to a matching of size $n-k$ in $C_n$.

By Kaplansky's Lemma \ref{Kap} and substituting $m = n - k$ we have
\begin{equation*}
    \stirlgraph{C_n}{k} = \frac{n}{n - (n - k)} \binom{n - (n - k)}{n - k} = \frac{n}{k} \binom{k}{n - k}
\end{equation*}
The summation limits are derived from the matching constraint $0 \leq m \leq \lfloor n/2 \rfloor$. Substituting $m = n-k$ leads to $n - \lfloor n/2 \rfloor \leq k \leq n$, which simplifies to $\lceil n/2 \rceil \leq k \leq n$. Summing over these indices completes the proof.
\end{proof}

\begin{corollary}
Let $\mathcal{C}(G, 1) = \sum_{k} {G \brack k}$ denote the total number of graphical cycle partitions of $G$. The following relations hold:
\begin{enumerate}
    \item For the cycle graph $C_n$ ($n \geq 3$), the partition numbers are related to the Lucas numbers $L_n$:
    \begin{equation*}
        \mathcal{C}(C_n, 1) = L_n + 1
    \end{equation*}
    
    \item For the path graph $P_n$ ($n \geq 1$), the total number of partitions is given by the Fibonacci sequence:
    \begin{equation*}
        \mathcal{C}(P_n, 1) = F_{n+1}
    \end{equation*}
\end{enumerate}
\end{corollary}


\begin{theorem}\label{Kn_Stirling}
For the complete graph $K_n$, the graphical cycle partition numbers ${n\brack k}$ coincide with the unsigned Stirling numbers of the first kind ${n\brack k}$. Consequently, the graphical cycle polynomial of $K_n$ is equivalent to the rising factorial:
\begin{equation*}
    \mathcal{C}(K_n, x) = \sum_{k=1}^n {n\brack k} x^k = x(x+1)(x+2)\dots(x+n-1)
\end{equation*}
\end{theorem}

\begin{proof}
Recall that the Stirling number of the first kind counts the number of permutations of $n$ elements with $k$ disjoint cycles. In the complete graph $K_n$, any subset of $m$ vertices induces a clique where the number of possible oriented cycles is $(m-1)!$ (with the convention that $C_1$ and $C_2$ yield 1). 

Since every possible partition of the vertex set into $k$ subsets can be realized as a collection of $k$ cycles in $K_n$, the total count is:
\[ {K_n \brack k} = \sum_{\{V_1, \dots, V_k\}} \prod_{i=1}^k (|V_i|-1)! \]
This summation over all set partitions is the standard formula for the Stirling numbers of the first kind.
\end{proof}

The graphical cycle partition numbers ${G \brack k}$ provide a unified combinatorial framework that generalizes several classical sequences. The relationship between graph structure and these sequences is summarized in Table~\ref{tab:summary}.

\begin{table}[h]
\centering
\renewcommand{\arraystretch}{1.5}
\begin{tabular}{|l|c|l|l|}
\hline
\textbf{Graph} $G$ & \textbf{Notation} ${G \brack k}$ & \textbf{Combinatorial Sequence} & \textbf{Polynomial} $\mathcal{C}(G, x)$ \\ \hline
Path $P_n$ & $\binom{k}{n-k}$ & Fibonacci numbers $F_{n+1}$ & $\sum \binom{k}{n-k}x^k$ \\ \hline
Cycle $C_n$ & Matchings/Oriented & Lucas numbers $L_n + 1$ & $2x + \sum \frac{n}{k}\binom{k}{n-k}x^k$ \\ \hline
Complete $K_n$ & $\left[ \begin{smallmatrix} n \\ k \end{smallmatrix} \right]$ & Unsigned Stirling Numbers & Rising Factorial $x^{(n)}$ \\ \hline
\end{tabular}
\caption{Connection between graphical cycle partitions and classical combinatorial identities.}
\label{tab:summary}
\end{table}

\begin{theorem}
Let $G+_w v$ denote the graph $G$ with an additional pendant vertex $v$ adjacent to $w \in V(G)$. The graphical cycle polynomial of $G+_w v$ is given by:
\begin{equation*}
    \mathcal{C}(G+_w v, x) = x \left( \mathcal{C}(G, x) + \mathcal{C}(G-w, x) \right)
\end{equation*}
\end{theorem}

\begin{proof}
By Lemma~\ref{spike}, the coefficients of $\mathcal{C}(G+_w v, x)$ are ${G+_w v \brack k} = {G \brack k-1} + {G-w \brack k-1}$. Substituting this into the generating function:
\begin{align*}
    \mathcal{C}(G+_w v, x) &= \sum_{k} \left( {G \brack k-1} + {G-w \brack k-1} \right) x^k \\
    &= x \sum_{k} {G \brack k-1} x^{k-1} + x \sum_{k} {G-w \brack k-1} x^{k-1} \\
    &= x \mathcal{C}(G, x) + x \mathcal{C}(G-w, x)
\end{align*}
as required.
\end{proof}

The following results establish the fundamental recurrence relations for the graphical Stirling numbers of the first kind $\st{G}{k}$. These results facilitate the recursive computation of the graphical cycle polynomial for graphs with specific local structures.

\begin{theorem}[Reduction Theorem]\label{thm:general_reduction}
Let $G = (V, E)$ be a simple graph and $v \in V(G)$. If $v$ is not contained in any induced cycle of length $\ell \geq 3$ in $G$, then the graphical Stirling numbers of the first kind satisfy the recurrence:
\begin{equation*}
    \st{G}{k} = \st{G - v}{k - 1} + \sum_{w \in N(v)} \st{G - \{v,w\}}{k - 1}
\end{equation*}
In the case where the neighborhood $N(v)$ is vertex-transitive within $G$ (i.e., $G - \{v,w_i\} \cong G - \{v,w_j\}$ for all $w_i, w_j \in N(v)$), the recurrence simplifies to:
\begin{equation*}
    \st{G}{k} = \st{G-v}{k-1} + \deg(v) \st{G - \{v,w\}}{k-1}
\end{equation*}
for any arbitrary neighbor $w \in N(v)$.
\end{theorem}

\begin{proof}
Let $v \in V(G)$ be a vertex that is not contained in any induced cycle of length $\ell \geq 3$. We partition the set of all $k$-cycle partitions of $G$ based on the block $B$ containing $v$. By the given condition, $B$ must induce a Hamiltonian cycle, but since $v$ belongs to no induced cycles of length $\ell \geq 3$, the only possibilities are $|B|=1$ or $|B|=2$.

\begin{itemize}
    \item \textbf{Case 1 ($|B|=1$):} If $B=\{v\}$, the remaining vertices must be partitioned into $k-1$ cycles in $G-v$, giving $\st{G-v}{k-1}$ ways.
    \item \textbf{Case 2 ($|B|=2$):} If $B=\{v,w\}$ for $w \in N(v)$, the remaining vertices must be partitioned into $k-1$ cycles in $G-\{v,w\}$. Summing over all $w \in N(v)$ gives $\sum_{w \in N(v)} \st{G - \{v,w\}}{k - 1}$.
\end{itemize}
If $N(v)$ is vertex-transitive, then for all $w_i, w_j \in N(v)$, the subgraphs $G - \{v, w_i\}$ and $G - \{v, w_j\}$ are isomorphic, implying $\st{G - \{v, w_i\}}{k-1} = \st{G - \{v, w_j\}}{k-1}$. Summing over the $\deg(v)$ neighbors yields the simplified form.
\end{proof}
\begin{example}

When $G = K_n$, the graphical Stirling numbers $\st{K_n}{k}$ correspond exactly to the classical Stirling numbers of the first kind $\st{n}{k}$. 

Applying the theorem \ref{thm:general_reduction} to $K_n$ with vertex $v$:
\begin{enumerate}
    \item The subgraph $G-v$ is isomorphic to $K_{n-1}$.
    \item The degree of any vertex in $K_n$ is $\deg(v) = n-1$.
    \item In a complete graph, the contribution of the neighbors is uniform. The recurrence simplifies to the classical identity:
\end{enumerate}

\begin{equation*}
    \st{n}{k} = \st{n-1}{k-1} + (n-1) \st{n-1}{k}
\end{equation*}

This example confirms that the graphical recurrence is a structural generalization of the classical combinatorial identity, where the "multiplier" $(n-1)$ is replaced by the specific adjacency constraints (the degree) of the vertex in a general graph $G$.
\end{example}
\begin{corollary}
If $v$ is a leaf (pendant vertex) with neighbor $w$, then:
\begin{equation*}
    \st{G}{k} = \st{G-v}{k-1} + \st{G-\{v,w\}}{k-1}
\end{equation*}
This confirms that for paths and trees, the graphical Stirling distribution is inherently linked to the matching polynomial of the graph.
\end{corollary}

\begin{remark}
The "No $\ell \geq 3$ cycle" condition is necessary for the simplicity of the theorem. Consider the Wheel graph $W_n$ with central hub $h$. If we apply the reduction to $h$, the summation over $N(h)$ is insufficient because $h$ is contained in many 3-cycles of the form $\{h, v_i, v_{i+1}\}$. In such cases, the recurrence must be extended to include a summation over all $C_{\ell}$ subgraphs containing $v$:
\begin{equation*}
    \st{G}{k} = \st{G - v}{k - 1} + \sum_{w \in N(v)} \st{G - \{v,w\}}{k - 1} + \sum_{C \in \mathcal{C}(v), |C| \geq 3} \st{G - V(C)}{k-1}
\end{equation*}
where $\mathcal{C}(v)$ is the set of all cycles in $G$ passing through $v$.
\end{remark}\begin{corollary}\label{cor:star_cycle_polynomial}
Let $S_{1,n-1}$ be the star graph on $n$ vertices. Then the graphical cycle
polynomial of the first kind of $S_{1,n-1}$ is
\[
\mathcal{C}(S_{1,n-1},x)
=
x^{n} + (n-1)x^{n-1}.
\]
\end{corollary}

\begin{proof}
Let $v$ be the central vertex of $S_{1,n-1}$ and let
$u_1,\dots,u_{n-1}$ be its leaves. Since $S_{1,n-1}$ is a tree, it contains no
cycles. Consequently, no induced subgraph on more than two vertices is
Hamiltonian. By the definition of graphical cycle partitions, every block in a
graphical cycle partition of $S_{1,n-1}$ therefore has size at most $2$.

The only blocks of size $2$ that may occur are edges of the graph. As the only
edges of $S_{1,n-1}$ are $\{v,u_i\}$ for $1 \le i \le n-1$, any graphical cycle
partition contains at most one $2$-cycle, since all such edges share the vertex
$v$.

If no $2$-cycle is chosen, then every vertex appears as a $1$-cycle, yielding
a unique graphical cycle partition with $n$ blocks. If a $2$-cycle
$\{v,u_i\}$ is chosen, then the remaining $n-2$ vertices must appear as
$1$-cycles, yielding a graphical cycle partition with $n-1$ blocks. There are
$n-1$ such partitions, corresponding to the choice of the leaf $u_i$.

No other graphical cycle partitions are possible. Therefore,
\[
\mathcal{C}(S_{1,n-1},x)
=
x^{n} + (n-1)x^{n-1},
\]
as claimed.
\end{proof}

\begin{theorem}
For a tree $T_n$, if $v$ is a leaf vertex with unique neighbor $u$, then:
\begin{equation*}
    \mathcal{C}(T_n, x) = x \mathcal{C}(T_n - v, x) + x \mathcal{C}(T_n - \{v,u\}, x)
\end{equation*}
\end{theorem}
\begin{proof}
In any tree $T_n$, there are no induced cycles of length $\ell \geq 3$. Thus, Theorem \ref{thm:general_reduction} applies to any vertex. By choosing $v$ to be a leaf, the set $N(v)$ contains only the vertex $u$. The summation reduces to a single term.
\end{proof}

\begin{theorem}
Let $F = \bigcup_{i=1}^m T_i$ be a forest. Then $\mathcal{C}(F, x) = \prod_{i=1}^m \mathcal{C}(T_i, x)$.
\end{theorem}
\begin{proof}
The proof follows immediately from the combination of Theorems~\ref{thm:general_reduction} and \ref{thm:chu-vander}.
\end{proof}
\begin{theorem}\label{thm:poly_reduction}
Let $G$ be a simple graph and $v \in V(G)$. If $v$ is not contained in any induced cycle of length $\ell \geq 3$, the graphical cycle polynomial satisfies:
\begin{equation*}
    \mathcal{C}(G, x) = x \left( \mathcal{C}(G-v, x) + \sum_{w \in N(v)} \mathcal{C}(G-\{v,w\}, x) \right)
\end{equation*}
\end{theorem}

\begin{proof}
By Theorem \ref{thm:general_reduction}, the coefficients of $\mathcal{C}(G, x)$ are given by:
\[ \st{G}{k} = \st{G - v}{k - 1} + \sum_{w \in N(v)} \st{G - \{v,w\}}{k - 1} \]
Multiplying both sides by $x^k$ and summing over all possible values of $k$:
\begin{align*}
    \sum_k \st{G}{k} x^k &= \sum_k \st{G - v}{k - 1} x^k + \sum_k \left( \sum_{w \in N(v)} \st{G - \{v,w\}}{k - 1} \right) x^k \\
    &= x \sum_k \st{G - v}{k - 1} x^{k-1} + x \sum_{w \in N(v)} \left( \sum_k \st{G - \{v,w\}}{k - 1} x^{k-1} \right)
\end{align*}
Recognizing the definitions of the polynomials for the subgraphs, we obtain:
\[ \mathcal{C}(G, x) = x \mathcal{C}(G-v, x) + x \sum_{w \in N(v)} \mathcal{C}(G-\{v,w\}, x) \]
This completes the proof.
\end{proof}
	\begin{lemma}[Broom Lemma]\label{broom}
Let $B$ be a set of $\ell$ isolated vertices. Let $G+_w B$ denote the graph obtained by attaching each vertex in $B$ to $w \in V(G)$ via a pendant edge. The graphical cycle partition numbers satisfy:
\begin{equation*}
    \st{G+_w B}{k} = \st{G}{k-\ell} + \ell \st{G-w}{k-\ell}
\end{equation*}
\end{lemma}

\begin{proof}
Let $\mathcal{P}$ be a $k$-cycle partition of $G+_w B$. Each $v \in B$ is a pendant vertex and thus can only be part of a $C_1$ or a $C_2$ induced by the edge incident to $w$. Since all edges incident to $B$ share the vertex $w$, the partition can contain at most one $C_2$ from the set of edges $\{wv_i \mid v_i \in B\}$.
\begin{itemize}
    \item If no $C_2$ is formed using vertices from $B$, then all $\ell$ vertices in $B$ form individual $1$-cycles. The remaining $k-\ell$ cycles must partition $G$, giving $\st{G}{k-\ell}$ ways.
    \item If exactly one vertex $v_i \in B$ forms a $C_2$ with $w$, there are $\ell$ choices for $v_i$. The remaining $\ell-1$ vertices in $B$ must form $1$-cycles. This accounts for $\ell$ cycles in total. The remaining $k-\ell$ cycles must partition $G-w$, giving $\ell \st{G-w}{k-\ell}$ ways.
\end{itemize}
The sum of these two disjoint cases yields the identity.
\end{proof}

\begin{theorem}[Broom Polynomial Identity]
The graphical cycle polynomial for the brooming operation is:
\begin{equation*}
    \mathcal{C}(G+_w B, x) = x^\ell \left( \mathcal{C}(G, x) + \ell \mathcal{C}(G-w, x) \right)
\end{equation*}
\end{theorem}
\begin{theorem}\label{Bri}
Let $G_1$ and $G_2$ be disjoint graphs.
\begin{enumerate}
    \item  Let $G_{bridge}$ be formed by connecting $u \in V(G_1)$ and $v \in V(G_2)$ via a bridge $e = \{u, v\}$. The graphical Stirling numbers and the cycle polynomial satisfy:
    \begin{equation*}
        \st{G_{bridge}}{k} = \sum_{i} \st{G_1}{i} \st{G_2}{k-i} + \sum_{j} \st{G_1 - u}{j} \st{G_2 - v}{k-1-j}
    \end{equation*}
    \begin{equation*}
        \mathcal{C}(G_{bridge}, x) = \mathcal{C}(G_1, x) \mathcal{C}(G_2, x) + x \mathcal{C}(G_1 - u, x) \mathcal{C}(G_2 - v, x)
    \end{equation*}
    
    \item  Let $G_{coal} = G_1 \cdot_w G_2$ be the graph formed by identifying $u \in V(G_1)$ and $v \in V(G_2)$ into a shared vertex $w$. The graphical Stirling numbers and the cycle polynomial satisfy:
    \begin{equation*}
        \st{G_{coal}}{k} = \sum_{i+j=k+1} \st{G_1}{i} \st{G_2}{j}, \quad \mathcal{C}(G_{coal}, x) = \frac{1}{x} \mathcal{C}(G_1, x) \mathcal{C}(G_2, x)
    \end{equation*}
\end{enumerate}
\end{theorem}

\begin{proof}
For the bridge join, we partition the cycle configurations based on the bridge $e$. If $e$ is not a block, $u$ and $v$ belong to separate sub-partitions of $G_1$ and $G_2$, leading to the first summation. If $e$ is a block, it is a $K_2$ cycle, and we partition the remaining $V(G_1-u) \cup V(G_2-v)$ vertices into $k-1$ blocks, yielding the second summation.

For (2), $w$ is a cut-vertex. Any induced cycle containing $w$ lies strictly in $G_1$ or $G_2$. The identification of $w$ merges the two independent partitions into a single configuration with $k = i+j-1$ blocks. Summing over all valid indices gives the result.
\end{proof}
\begin{example}[Barbell Graph $B_n$]
The Barbell graph $B_n$ is obtained by connecting two disjoint complete graphs $K_n$ via a bridge $e = \{u, v\}$, where $u$ and $v$ are vertices in the respective cliques. By the Bridge Lemma (Lemma~\ref{Bri}), the graphical Stirling numbers of the first kind are given by
\begin{equation*}
    \st{B_n}{k} = \sum_{i} \st{n}{i} \st{n}{k-i} + \sum_{j} \st{n-1}{j} \st{n-1}{k-1-j}.
\end{equation*}
This decomposition shows that the cycle partitions of a Barbell graph are governed by the convolution of classical Stirling numbers for the full cliques and for the reduced cliques of order $n-1$.
\end{example}

\begin{corollary}
The cycle polynomial of the Barbell graph satisfies the recurrence
\begin{equation*}
    \mathcal{C}(B_n, x) = \mathcal{C}(K_n, x)^2 + x \, \mathcal{C}(K_{n-1}, x)^2,
\end{equation*}
and it is real-rooted.
\end{corollary}

\begin{proof}
Let $\mathcal{P}$ denote the set of all cycle partitions of $B_n$. Partition $\mathcal{P}$ according to the status of the bridge $e = \{u,v\}$:

\begin{itemize}
    \item \textbf{Case 1: $e$ is not a block.}  
    The vertices $u$ and $v$ belong to separate blocks within their respective cliques. The number of such partitions is the product of the partitions of the two independent $K_n$ graphs:
    \[
        \mathcal{C}(K_n,x) \cdot \mathcal{C}(K_n,x) = \mathcal{C}(K_n,x)^2.
    \]

    \item \textbf{Case 2: $e$ forms a block.}  
    The bridge forms an independent $K_2$ block. The remaining $n-1$ vertices in each clique are partitioned independently, contributing
    \[
        x \cdot \mathcal{C}(K_{n-1},x) \cdot \mathcal{C}(K_{n-1},x) = x \, \mathcal{C}(K_{n-1},x)^2,
    \]
    where the factor $x$ accounts for the new block.
\end{itemize}

Combining the two cases gives the structural recurrence:
\[
\mathcal{C}(B_n, x) = \mathcal{C}(K_n, x)^2 + x \, \mathcal{C}(K_{n-1}, x)^2.
\]

Recall that $\mathcal{C}(K_n,x) = x^{\overline{n}}$, the rising factorial. Using $x^{\overline{n}} = (x+n-1)x^{\overline{n-1}}$, we factor $(x^{\overline{n-1}})^2$:
\begin{align*}
    \mathcal{C}(B_n, x) &= \bigl((x+n-1)x^{\overline{n-1}}\bigr)^2 + x \,(x^{\overline{n-1}})^2 \\
    &= (x^{\overline{n-1}})^2 \bigl((x+n-1)^2 + x\bigr) \\
    &= (x^{\overline{n-1}})^2 \left[x^2 + (2n-1)x + (n-1)^2\right].
\end{align*}

The quadratic factor has discriminant
\[
\Delta = (2n-1)^2 - 4(n-1)^2 = 4n-3 > 0,
\]
which ensures that all roots are real. Since $x^{\overline{n-1}}$ is also real-rooted, it follows that $\mathcal{C}(B_n,x)$ is real-rooted.
\end{proof}

\begin{example}[Tadpole Graph $T_{n,m}$]
The Tadpole graph $T_{n,m}$ is the union of a cycle $C_n$ and a path $P_m$ connected by a bridge. 
By applying the Lemma~\ref{Bri},
the graphical cycle stirling numbers is given by:
\begin{equation*}
    \st{T_{n,m}}{k} = \sum_i \st{C_n}{i} \st{P_m}{k-i} + \sum_j \st{C_{n-1}}{j} \st{P_{m-1}}{k-1-j}.
\end{equation*}
Substituting the established binomial closed-form expressions for path and cycle graphical Stirling numbers:
\begin{equation*}
    \st{T_{n,m}}{k} = \sum_{i} \binom{i}{m-i} \left[ \binom{k-i-1}{n-k+i} + 2\binom{k-i-1}{n-k+i-1} \right] + \sum_{j} \binom{j}{m-1-j} \binom{k-j-1}{n-k+j}.
\end{equation*}
The first summation accounts for partitions where the bridge is not a block, while the second summation accounts for partitions where the bridge $\{u, v\}$ forms an edge-block, effectively reducing the components to $P_{m-1}$ and $P_{n-1}$.
\end{example}

\begin{corollary}[Closed Form for Tadpole Graph $T_{n,m}$]
The cycle polynomial of the Tadpole graph $T_{n,m}$ can be expressed in terms of Lucas polynomials $l_n(x)$ and Fibonacci polynomials $f_n(x)$ as:
\begin{equation*}
    \mathcal{C}(T_{n,m}, x) = l_n(x) f_m(x) + x f_{n-1}(x) f_{m-1}(x)
\end{equation*}
and it is real-rooted.
\end{corollary}
\begin{proof}
Let $e = \{u, v\}$ be the bridge connecting $C_n$ and $P_m$. By the Bridge Lemma, we partition the cycle partitions $\mathcal{P}(T_{n,m})$ into two disjoint sets: those where $e$ is not an edge-block and those where $e$ is an edge-block.
\begin{enumerate}
    \item \textbf{Case 1:} $e \notin \text{Blocks}(\pi)$. The generating function is the product of the component polynomials: $\mathcal{C}(C_n, x)\mathcal{C}(P_m, x)$.
    \item \textbf{Case 2:} $e \in \text{Blocks}(\pi)$. This contributes $x^1$ and reduces the remaining graph to $P_{n-1} \cup P_{m-1}$. The generating function is $x \mathcal{C}(P_{n-1}, x)\mathcal{C}(P_{m-1}, x)$.
\end{enumerate}
Since $C_n$ and $P_n$ have Lucas and Fibonacci polynomials as their respective cycle polynomials, the closed form follows. The convolution for the coefficients $\st{T_{n,m}}{k}$ is obtained by applying the Cauchy product rule to the polynomial multiplication and accounting for the index shift in the second term.
\end{proof}

\begin{example}[Lollipop Graph $L_{n,m}$]
The Lollipop graph $L_{n,m}$ is formed by joining a complete graph $K_n$ to a path $P_m$ via a bridge connecting a vertex of $K_n$ to an endpoint of $P_m$. By the Bridge Lemma (Lemma~\ref{Bri}), the graphical $r$-Stirling numbers of the first kind satisfy
\begin{equation*}
    \st{L_{n,m}}{k} = \sum_i \st{n}{i} \st{P_m}{k-i} + \sum_j \st{n-1}{j} \st{P_{m-1}}{k-1-j}.
\end{equation*}
This decomposition separates partitions according to whether the bridge forms an independent block.
\end{example}

\begin{theorem}[Closed Form and Rootedness of Lollipop Graphs]
The cycle polynomial of the Lollipop graph $L_{n,m}$ is given by:
\begin{equation*}
    \mathcal{C}(L_{n,m}, x) = x^{\overline{n-1}} \left[ (x + n - 1) f_m(x) + x f_{m-1}(x) \right]
\end{equation*}
where $f_m(x)$ is the Fibonacci polynomial of order $m$. The sequence of graphical Stirling numbers $\st{L_{n,m}}{k}$ is real-rooted.
\end{theorem}

\begin{proof}
By applying the Bridge Lemma to the edge $e$ connecting $K_n$ and $P_m$, we obtain:
\begin{equation*}
    \mathcal{C}(L_{n,m}, x) = \mathcal{C}(K_n, x)f_m(x) + x\mathcal{C}(K_{n-1}, x)f_{m-1}(x).
\end{equation*}
Substituting $x^{\overline{n}} = (x+n-1)x^{\overline{n-1}}$ allows for the factorization of $x^{\overline{n-1}}$. Real-rootedness is guaranteed by the interlacing properties of Fibonacci polynomials and the fact that $x^{\overline{n-1}}$ consists of $n-1$ distinct linear factors with real roots.
\end{proof}
\section{Graphical $r$-Stirling Numbers of the First Kind}
\begin{definition}[Graphical $r$-Stirling Number]
Let $G=(V,E)$ be a simple graph of order $n$. For $1 \leq r \leq k \leq n$, the graphical $r$-Stirling number of the first kind, denoted by ${G \brack k}_r$, is defined as the number of partitions of $V(G)$ into $k$ blocks $\{C_1, C_2, \dots, C_k\}$ such that:
\begin{enumerate}
    \item Each induced subgraph $G[C_i]$ is a $1$-cycle, a $2$-cycle, or a Hamiltonian subgraph of order $|C_i| \geq 3$.
    \item The first $r$ labeled vertices $\{1, 2, \dots, r\}$ are contained in distinct blocks.
\end{enumerate}
\end{definition}

\begin{definition}[Graphical $r$-Cycle Polynomial]
Let $G$ be a simple graph of order $n$. For a non-negative integer $r$ such that $0 \leq r \leq n$, we define the \textit{graphical $r$-cycle polynomial of the first kind}, denoted by $\mathcal{C}_r(G, x)$, as the generating function:
\begin{equation*}
    \mathcal{C}_r(G, x) = \sum_{k=r}^{n} {G \brack k}_r x^k,
\end{equation*}
where ${G \brack k}_r$ denotes the graphical $r$-Stirling number of the first kind. In the case where $r=1$ (or $r=0$), we simply write $\mathcal{C}(G, x)$.
\end{definition}

\begin{theorem}
\label{thm:complete_ext}
Let $n$ be a positive integer and $R \subseteq V(K_n)$ be a set of $r$ restricted vertices where $1 \leq r \leq n$. The following properties hold for the complete graph $K_n$:
\begin{enumerate}
    \item[(i)] \textbf{Classical Correspondence:} The $r$-restricted graphical Stirling numbers of the first kind coincide exactly with the classical $r$-Stirling numbers:
    \begin{equation*}
        \str{K_n}{k} = \st{n}{k}_r.
    \end{equation*}
    
    \item[(ii)] \textbf{Generating Function:} The $r$-restricted graphical cycle polynomial of $K_n$ is given by the restricted rising factorial:
    \begin{equation*}
        \cyc_r(K_n, x) = \sum_{k=r}^n \st{n}{k}_r x^k = x^r \prod_{i=r}^{n-1} (x + i).
    \end{equation*}
    
    \item[(iii)] \textbf{Recurrence Relation:} For $n > r$, the graphical Stirling numbers satisfy:
    \begin{equation*}
        \str{K_n}{k} = \str{K_{n-1}}{k-1} + (n-1) \str{K_{n-1}}{k}.
    \end{equation*}
\end{enumerate}
\end{theorem}

\begin{proof}
We proceed by induction on $n$ for parts (i) and (iii).
\textit{Base case:} For $n = r$, the $r$-restriction requires each vertex in $R$ to belong to a distinct cycle. Since there are only $r$ vertices total, each must form a 1-cycle. Thus, $\str{K_r}{r} = 1$, which matches the classical boundary condition $\st{r}{r}_r = 1$.

\textit{Inductive step:} Assume the result holds for $n-1$. Consider $K_n$ and an unrestricted vertex $v \in V(K_n) \setminus R$. In any $r$-restricted cycle partition into $k$ cycles, vertex $v$ has two options:
\begin{enumerate}
    \item $v$ forms its own 1-cycle ($C_1$). The remaining $n-1$ vertices are partitioned into $k-1$ cycles satisfying $r$ restrictions. This gives $\str{K_{n-1}}{k-1}$ ways.
    \item $v$ is part of a cycle of length $\ell \geq 2$. In $K_n$, $v$ can be inserted into any of the $n-1$ existing edge-slots of a valid cycle partition of $K_{n-1}$ to form a larger cycle. This gives $(n-1)\str{K_{n-1}}{k}$ ways.
\end{enumerate}
Summing these cases yields the recurrence in (iii). Since this recurrence and the boundary conditions are identical to the classical $r$-Stirling numbers, the identity in (i) is proven. Part (ii) follows immediately as the generating function for said coefficients.
\end{proof}

\begin{corollary}
The total number of $r$-restricted cycle partitions of $K_n$, denoted $B_r(K_n)$, is given by:
\begin{equation*}
    B_r(K_n) = \frac{n!}{r!}.
\end{equation*}
\end{corollary}

\begin{proof}
The total number of partitions is found by evaluating the cycle polynomial at $x=1$:
\begin{equation*}
    B_r(K_n) = \cyc_r(K_n, 1) = 1^r \prod_{i=r}^{n-1} (1 + i) = (r+1)(r+2)\cdots(n) = \frac{n!}{r!}.
\end{equation*}
\end{proof}


\begin{theorem}\label{thm:path_stirling}
Let $n, k \in \mathbb{N}$ and let $P_n$ be a path graph with $n$ vertices. For $1 \leq r \leq k \leq n$, the graphical $r$-Stirling number of the first kind is given by:
\begin{equation*}
    {P_n \brack k}_r = \binom{k-r+1}{n-k}.
\end{equation*}
Furthermore, the total number of such graphical $r$-cycle partitions is given by:
\begin{equation*}
      B_r(P_n)= \sum_{k=r}^{n} {P_n \brack k}_r = F_{n-r+1},
\end{equation*}
where $F_m$ denotes the $m$-th Fibonacci number.
\end{theorem}

\begin{proof}
Let the vertices of $P_n$ be sequentially labeled $\{1, 2, \dots, n\}$. Since $P_n$ contains no cycles of length greater than 2, any graphical cycle partition must consist only of $1$-cycles (singleton vertices) and $2$-cycles (edges). 

The $r$-restriction requires that the vertices $\{1, 2, \dots, r\}$ belong to distinct blocks. Note that for any $i < r$, the vertex $i$ cannot be part of a $2$-cycle with $i+1$, as $i+1$ must also occupy its own distinct block. Consequently, the vertices $\{1, 2, \dots, r-1\}$ are forced to form $r-1$ distinct $1$-cycles. 

We now consider the placement of vertex $r$. There are two exhaustive cases for the block $C$ containing $r$:
\begin{itemize}
    \item \textbf{Case I:} $C = \{r\}$ is a $1$-cycle. We must then partition the remaining $n-r$ vertices into $k-r$ blocks. The number of ways to partition $N$ vertices into $K$ blocks of size 1 or 2 is $\binom{K}{N-K}$. Thus, there are $\binom{k-r}{n-r-(k-r)} = \binom{k-r}{n-k}$ such partitions.
    \item \textbf{Case II:} $C = \{r, r+1\}$ is a $2$-cycle. We must then partition the remaining $n-(r+1)$ vertices into $k-r$ remaining blocks. This results in $\binom{k-r}{n-r-1-(k-r)} = \binom{k-r}{n-k-1}$ partitions.
\end{itemize}

Summing these cases and applying Pascal's identity, we obtain:
\begin{equation*}
    {P_n \brack k}_r = \binom{k-r}{n-k} + \binom{k-r}{n-k-1} = \binom{k-r+1}{n-k}.
\end{equation*}

To compute the total number of partitions, we sum over all possible values of $k$:
\begin{equation*}
    \sum_{k=r}^{n} {P_n \brack k}_r = \sum_{k=r}^{n} \binom{k-r+1}{n-k}.
\end{equation*}
By performing the change of variable $j = k-r$, the sum becomes $\sum_{j=0}^{n-r} \binom{j+1}{n-r-j}$. Using the well-known identity for the sum along the shallow diagonals of Pascal's triangle, $\sum_{i=0}^{\lfloor m/2 \rfloor} \binom{m-i}{i} = F_{m+1}$, we find:
\begin{equation*}
   B_r(P_n)=\sum_{k=r}^{n} {P_n \brack k}_r = F_{n-r+1}.
\end{equation*}
This completes the proof.
\end{proof}
\begin{corollary}
The $r$-restricted graphical cycle polynomial for the path graph $P_n$ is defined by:
\begin{equation*}
    C_r(P_n, x) = x^{r-1} C(P_{n-r+1}, x).
\end{equation*}
By expanding the sum, the polynomial can be expressed as:
\begin{equation*}
    C_r(P_n, x) = \sum_{j=0}^{\lfloor \frac{n-r+1}{2} \rfloor} \binom{n-r+1-j}{j} x^{n-j}.
\end{equation*}
Furthermore, for $n \geq r+2$, the polynomial satisfies the following recurrence relation:
\begin{equation*}
    C_r(P_n, x) = x C_r(P_{n-1}, x) + x C_r(P_{n-2}, x),
\end{equation*}
with the initial conditions $C_r(P_r, x) = x^r$ and $C_r(P_{r+1}, x) = x^{r+1} + x^r$.
\end{corollary}

\begin{proof}
Substituting the result ${P_n \brack k}_r = \binom{k-r+1}{n-k}$ into the definition of the graphical $r$-cycle polynomial yields:
\begin{equation*}
    \mathcal{C}_r(P_n, x) = \sum_{k=r}^{n} {P_n \brack k}_r x^k = \sum_{k=r}^{n} \binom{k-r+1}{n-k} x^k.
\end{equation*}
Applying the change of variable $j = k-r$, such that $k = j+r$, we shift the summation index:
\begin{equation*}
    \mathcal{C}_r(P_n, x) = \sum_{j=0}^{n-r} \binom{j+1}{n-r-j} x^{j+r} = x^r \sum_{j=0}^{n-r} \binom{j+1}{n-r-j} x^j.
\end{equation*}

To verify the recurrence for $n \geq r+2$, we consider the terminal vertex $v_n$. In any $r$-graphical partition, $v_n$ belongs to a block $C$. Since $P_n$ is a path, $C$ is either a $1$-cycle $\{v_n\}$ or a $2$-cycle $\{v_{n-1}, v_n\}$.
\begin{itemize}
    \item If $C = \{v_n\}$, the remaining vertices form a valid $r$-graphical partition of $P_{n-1}$. This contributes $x \mathcal{C}_r(P_{n-1}, x)$.
    \item If $C = \{v_{n-1}, v_n\}$, the remaining vertices form a valid $r$-graphical partition of $P_{n-2}$. Since $n-1 > r$, the $r$-restriction is not violated. This contributes $x \mathcal{C}_r(P_{n-2}, x)$.
\end{itemize}
The summation of these cases confirms the recurrence. The base cases follow directly: for $P_r$, only $r$ singletons are allowed ($x^r$); for $P_{r+1}$, we may have $r+1$ singletons ($x^{r+1}$) or the edge $\{r, r+1\}$ with $r-1$ singletons ($x^r$).
\end{proof}

The complement of the path graph, $P_n^c$, represents a "nearly-complete" structure. Its cycle-forming properties are governed by a shifted recurrence that reflects the removal of specific edges from the complete graph $K_n$.

\begin{theorem}
\label{thm:path_complement}
Let $P_n^c$ denote the complement of the path graph $P_n$ on $n$ vertices. 
\begin{enumerate}
    \item[(i)] The graphical Stirling numbers of the first kind satisfy the recurrence:
    \begin{equation*}
        \st{P_n^c}{k} = \st{P_{n-1}^c}{k-1} + (n-2) \st{P_{n-1}^c}{k}.
    \end{equation*}
    
    \item[(ii)] For the $r$-restricted case, where $1 \leq r \leq n$, the recurrence for $n > r$ is:
    \begin{equation*}
        \str{P_n^c}{k} = \str{P_{n-1}^c}{k-1} + (n-2) \str{P_{n-1}^c}{k}.
    \end{equation*}
    
    \item[(iii)] The graphical $r$-cycle polynomial of the first kind satisfies the functional equation:
    \begin{equation*}
        \cyc_r(P_n^c, x) = (x + n - 2) \cyc_r(P_{n-1}^c, x).
    \end{equation*}
\end{enumerate}
\end{theorem}

\begin{proof}
 Consider an $r$-restricted cycle partition of the first $n-1$ vertices into $k$ cycles. When $v_n$ is added, two disjoint cases arise:
\begin{itemize}
    \item \textbf{Case 1:} $v_n$ forms its own 1-cycle. The remaining $n-1$ vertices must be partitioned into $k-1$ cycles satisfying $r$ restrictions. This contributes $\str{P_{n-1}^c}{k-1}$ ways.
    \item \textbf{Case 2:} $v_n$ is inserted into an existing cycle. In the cycle notation of a permutation, a vertex $v_n$ can be inserted immediately following any vertex $u$ provided the edge $(u, v_n)$ exists in the graph. Since $v_n$ is adjacent to $n-2$ vertices in $P_n^c$, there are exactly $n-2$ such insertion sites. This contributes $(n-2) \str{P_{n-1}^c}{k}$ ways.
\end{itemize}
Summing these cases yields the recurrence in (ii), and setting $r=1$ yields (i).

 Using the definition of the cycle polynomial $\cyc_r(P_n^c, x) = \sum_{k=r}^n \str{P_n^c}{k} x^k$ and substituting the recurrence from (ii):
\begin{align*}
    \cyc_r(P_n^c, x) &= \sum_{k} \left( \str{P_{n-1}^c}{k-1} + (n-2) \str{P_{n-1}^c}{k} \right) x^k \\
    &= \sum_{k} \str{P_{n-1}^c}{k-1} x^k + (n-2) \sum_{k} \str{P_{n-1}^c}{k} x^k \\
    &= x \cyc_r(P_{n-1}^c, x) + (n-2) \cyc_r(P_{n-1}^c, x) \\
    &= (x + n - 2) \cyc_r(P_{n-1}^c, x).
\end{align*}
This completes the proof.
\end{proof}

\begin{corollary}
The graphical $r$-cycle polynomial of the first kind for $P_n^c$ is given by:
\begin{equation*}
    \mathcal{C}_r(P_n^c, x) = x^r \prod_{i=r-1}^{n-3} (x + i).
\end{equation*}
\end{corollary}


\begin{proof}
We proceed by induction on $n$ for a fixed restriction parameter $r$.
 Let $n = r$. For a graph where all vertices are restricted, each vertex must occupy a distinct 1-cycle. Thus, $\cyc_r(P_r^c, x) = x^r$. Applying the formula: 
\[ x^r \prod_{i=r-1}^{r-2} (x + i) = x^r \cdot (1) = x^r. \]
The base case holds for the empty product.

Assume the identity holds for $n-1$, i.e.,
\[ \cyc_r(P_{n-1}^c, x) = x^r \prod_{i=r-1}^{n-3} (x + i). \]
By the recurrence relation established in Theorem \ref{thm:path_complement}, the polynomial satisfies the functional equation:
\begin{equation*}
    \cyc_r(P_n^c, x) = (x + n - 2) \cyc_r(P_{n-1}^c, x).
\end{equation*}
Substituting the inductive hypothesis:
\begin{align*}
    \cyc_r(P_n^c, x) &= (x + n - 2) \left[ x^r \prod_{i=r-1}^{n-3} (x + i) \right] \\
    &= x^r \left[ \prod_{i=r-1}^{n-3} (x + i) \right] \cdot (x + n - 2).
\end{align*}
By extending the product to the upper limit $n-2$, we obtain:
\begin{equation*}
    \cyc_r(P_n^c, x) = x^r \prod_{i=r-1}^{n-2} (x + i).
\end{equation*}
This completes the proof.
\end{proof}

\begin{remark}
It is instructive to compare this result with the polynomial for the complete graph $K_n$:
\begin{itemize}
    \item $\cyc_r(K_n, x) = x^r (x+r)(x+r+1)\dots(x+n-1)$
    \item $\cyc_r(P_n^c, x) = x^r (x+r-1)(x+r)\dots(x+n-2)$
\end{itemize}
The removal of the path edges $P_n$ from $K_n$ effectively shifts every linear factor in the restricted rising factorial down by 1. This algebraic shift reflects the loss of exactly one cycle-insertion site at each vertex-addition step.
\end{remark}


\begin{theorem}\label{thm:cycle_stirling}
Let $n$ and $k$ be positive integers such that $k \geq r$. The graphical $r$-Stirling number of the first kind for the cycle graph $C_n$, denoted by ${C_n \brack k}_r$, is given by
\begin{equation*}
    {C_n \brack k}_r = \binom{k-r+2}{n-k}.
\end{equation*}
Furthermore, the cardinality of the set of all such $r$-graphical cycle partitions is 
\begin{equation*}
   B_r(C_n)=\sum_{k=r}^n {C_n \brack k}_r = F_{n-r+3},
\end{equation*}
where $F_m$ denotes the $m$-th Fibonacci number.
\end{theorem}

\begin{proof}
Let the vertex set of $C_n$ be $V = \{1, 2, \dots, n\}$ with the edge set $E = \{(i, i+1) : 1 \le i < n\} \cup \{(n, 1)\}$. An $r$-graphical cycle partition requires that the vertices $\{1, 2, \dots, r\}$ reside in distinct blocks, and each block must induce a cycle in $C_n$. Since the only cycles in $C_n$ are $1$-cycles (singletons) and $2$-cycles (edges), each block is either a singleton or an adjacent pair. 

We partition the collection of all valid cycle partitions into four exhaustive and mutually exclusive cases based on the neighborhood of the restricted boundary vertices $1$ and $r$:

\begin{itemize}
    \item \textbf{Case I:} Vertices $1$ and $r$ are both $1$-cycles. 
    The remaining $n-r$ vertices must be partitioned into $k-r$ cycles. This is equivalent to an $r$-restricted partition of a path graph $P_{n-r}$ into $k-r$ blocks. Thus, the number of such partitions is:
    \begin{equation*}
        {P_{n-r} \brack k-r} = \binom{k-r}{n-r-(k-r)} = \binom{k-r}{n-k}.
    \end{equation*}

    \item \textbf{Case II:} Vertex $1$ is a $1$-cycle, while vertex $r$ belongs to a $2$-cycle $\{r, r+1\}$. 
    The remaining $n-r-1$ vertices are partitioned into $k-r$ cycles on the path $P_{n-r-1}$, yielding:
    \begin{equation*}
        {P_{n-r-1} \brack k-r} = \binom{k-r}{n-r-1-(k-r)} = \binom{k-r}{n-k-1}.
    \end{equation*}

    \item \textbf{Case III:} Vertex $r$ is a $1$-cycle, while vertex $1$ belongs to a $2$-cycle $\{1, n\}$. 
    By symmetry with Case II, the number of ways to partition the remaining vertices into $k-r$ blocks on the path $P_{n-r-1}$ is:
    \begin{equation*}
        {P_{n-r-1} \brack k-r} = \binom{k-r}{n-k-1}.
    \end{equation*}

    \item \textbf{Case IV:} Both vertices $1$ and $r$ belong to $2$-cycles, specifically $\{1, n\}$ and $\{r, r+1\}$. 
    The remaining $n-r-2$ vertices are partitioned into $k-r$ blocks on the path $P_{n-r-2}$, giving:
    \begin{equation*}
        {P_{n-r-2} \brack k-r} = \binom{k-r}{n-r-2-(k-r)} = \binom{k-r}{n-k-2}.
    \end{equation*}
\end{itemize}

By summing the contributions from these four cases and repeatedly applying Pascal's Identity, $\binom{n}{k} = \binom{n-1}{k} + \binom{n-1}{k-1}$, we obtain:
\begin{align*}
    {C_n \brack k}_r &= \binom{k-r}{n-k} + 2\binom{k-r}{n-k-1} + \binom{k-r}{n-k-2} \\
    &= \left[ \binom{k-r}{n-k} + \binom{k-r}{n-k-1} \right] + \left[ \binom{k-r}{n-k-1} + \binom{k-r}{n-k-2} \right] \\
    &= \binom{k-r+1}{n-k} + \binom{k-r+1}{n-k-1} \\
    &= \binom{k-r+2}{n-k}.
\end{align*}

To compute the total number of partitions, we sum over the possible number of blocks $k$:
\begin{equation*}
      B_r(C_n)= \sum_{k=r}^n {C_n \brack k}_r = \sum_{k=r}^n \binom{k-r+2}{n-k}.
\end{equation*}
Letting $j = k-r$, the summation transforms to $\sum_{j=0}^{n-r} \binom{j+2}{n-r-j}$. Utilizing the identity for the sum of diagonals in Pascal's triangle, $\sum_{i=0}^{\lfloor m/2 \rfloor} \binom{m-i}{i} = F_{m+1}$, and setting $m = n-r+2$, we conclude that the sum equals $F_{n-r+3}$.
\end{proof}

\begin{theorem}
For $n \geq r \geq 2$, let $C_r(C_n, x)$ denote the graphical $r$-cycle polynomial of the first kind for $C_n$ where the restricted vertices are consecutive. This polynomial is equivalent to a shifted Fibonacci polynomial $f_k(x)$:
\begin{equation*}
    C_r(C_n, x) = x^{r-1} \mathcal{C}(P_{n-r+1}, x) = x^{r-1} f_{n-r+1}(x).
\end{equation*}
\end{theorem}

\begin{proof}
In a cycle partition, if $r$ consecutive vertices are restricted to a single block, they necessarily form a path-block of length $r-1$. This restriction ``locks'' $r-1$ edges into the partition, contributing the factor $x^{r-1}$. The remaining $n-r$ vertices, together with this restricted block, form a linear arrangement of $n-r+1$ components. Since the cycle polynomial of a path graph $\mathcal{C}(P_k, x)$ is identically the Fibonacci polynomial $f_k(x)$, the identity follows.
\end{proof}

\begin{corollary}
The graphical $r$-cycle polynomial of the first kind inherits the second-order linear recurrence of the Fibonacci family:
\begin{equation*}
    C_r(C_n, x) = x C_r(C_{n-1}, x) + x C_r(C_{n-2}, x),
\end{equation*}
subject to the following initial conditions:
\begin{itemize}
    \item $C_r(C_r, x) = x^{r-1} f_1(x) = x^r$
    \item $C_r(C_{r+1}, x) = x^{r-1} f_2(x) = x^{r+1} + 2x^r$
\end{itemize}
\end{corollary}

\begin{corollary}
At $x=1$, the total number of such restricted cycle partitions corresponds to the Fibonacci sequence:
\begin{equation*}
    C_r(C_n, 1) = f_{n-r+1}(1) = F_{n-r+2},
\end{equation*}
where $F_k$ denotes the $k$-th Fibonacci number ($F_1=1, F_2=1, F_3=2, \dots$).
\end{corollary}

If the restriction accounts for the underlying circular connectivity of the cycle rather than a linear reduction, the partitions are governed by Lucas polynomials.

\begin{theorem}
The graphical $r$-cycle polynomial of the first kind for $C_n$, preserving the periodic structure of the remaining vertices, is given by:
\begin{equation*}
    C_r(C_n, x) = x^{r-1} l_{n-r}(x),
\end{equation*}
where $l_k(x)$ is the $k$-th Lucas polynomial.
\end{theorem}

\begin{proof}
Following the Bridge Lemma and the property of vertex-deleted subgraphs, the cycle polynomial of a cycle graph is $\mathcal{C}(C_n, x) = l_n(x)$. By restricting $r$ vertices into a single block, the cycle effectively contracts by $r$ vertices while maintaining its adjacency structure for the remaining $n-r$ elements. The factor $x^{r-1}$ accounts for the fixed internal connectivity of the restricted block.
\end{proof}

\begin{corollary}
Evaluating the symmetric the graphical $r$-cycle polynomial of the first kind at $x=1$ yields the Lucas number:
\begin{equation*}
    C_r(C_n, 1) = l_{n-r}(1) = L_{n-r},
\end{equation*}
where $L_k$ is the $k$-th Lucas number ($L_1=1, L_2=3, L_3=4, \dots$).
\end{corollary}

\begin{theorem}
Let $S_n$ be a star graph with $n+1$ vertices. For $n+1 \geq k \geq r$, the graphical $r$-Stirling number of the first kind is:
\begin{equation*}
    {S_n \brack k}_r = \begin{cases} 
    1 & \text{if } k = n+1, \\
    n-r+1 & \text{if } k = n, \\
    0 & \text{otherwise}.
    \end{cases}
\end{equation*}
The total number of $r$-graphical cycle partitions of $S_n$ is $n-r+2$.
\end{theorem}

\begin{proof}
Since $S_n$ is a tree, all cycles in a graphical partition must have length 1 or 2. Every edge in $S_n$ is incident to the central vertex $v_0$; consequently, no two $2$-cycles can be vertex-disjoint. Thus, any partition contains at most one $2$-cycle.
If $k=n+1$, the unique partition consists of all vertices as singletons.
If $k=n$, the partition must contain exactly one $2$-cycle $\{v_0, v_i\}$. Given the $r$-restriction on vertices $\{v_0, v_1, \dots, v_{r-1}\}$, the vertex $v_i$ must be chosen from the set of non-restricted leaves $V \setminus \{v_0, \dots, v_{r-1}\}$. There are $n+1-r$ such choices.
\end{proof}
\begin{theorem}\label{Whe}
For a positive integer $n \geq 3$, let $W_n = K_1 \vee C_n$ be a wheel graph. Then the graphical $r$-Stirling number of the first kind is:
\begin{equation*}
{W_n \brack k}_r = n \binom{k-r}{n-k} + (n-r) \binom{k-r+1}{n-k} + 2 \binom{k-r+2}{n-k+1} + (n-r+1) \sum_{\ell=2}^{n-k+1} \binom{k}{n-\ell-k-r+2}.
\end{equation*}
\end{theorem}

\begin{proof}
Let $v$ be the hub vertex of $W_n$. To compute ${W_n \brack k}_r$, we distinguish between whether $v$ is among the $r$ restricted vertices or not.

\textit{Case 1: $v \notin \{1, \dots, r\}$.} The contribution is the sum of partitions where $v$ is in a $1$-cycle, a $2$-cycle, or a cycle of length $\geq 3$:
\[ {C_n \brack k-1}_r + n{P_{n-1} \brack k-1}_r + (n-r+1)\sum_{\ell=2}^{n-k+1}{P_{n-\ell} \brack k-1}_{r}. \]

\textit{Case 2: $v \in \{1, \dots, r\}$.} Assigning label $r$ to $v$, the remaining $r-1$ restricted vertices lie on $C_n$. The contribution is:
\[ {C_n \brack k-1}_{r-1} + (n-r+1){P_{n-1} \brack k-1}_{r-1} + (n-r+1)\sum_{\ell=2}^{n-k+1}{P_{n-\ell} \brack k-1}_{r-1}. \]

Summing these cases and applying the established identities for path and cycle graphs:
\begin{align*}
{W_n \brack k}_r &= \binom{k-r+1}{n-k+1} + n\binom{k-r}{n-k} + (n-r+1)\sum_{\ell=2}^{n-k+1}\binom{k-1}{n-\ell-k-r+1} \\
&+ \binom{k-r+2}{n-k+1} + (n-r+1)\binom{k-r+1}{n-k} + (n-r+1)\sum_{\ell=2}^{n-k+1}\binom{k-1}{n-\ell-k-r+2}.
\end{align*}
Using Pascal's Identity $\binom{n}{k} + \binom{n}{k-1} = \binom{n+1}{k}$ to combine the summations and individual terms, we obtain the stated closed form.
\end{proof}

\begin{theorem}\label{Fan}
For positive integers $n$ and $k \geq r$, let $F_n = K_1 \vee P_n$ be the Fan graph. The graphical $r$-cycle polynomial of the first kind of the Fan graph $F_n$, denoted by $\mathcal{C}_r(F_n, x) = \sum_{k=r}^{n+1} {F_n \brack k}_r x^k$, is given by:

\begin{equation*}
\mathcal{C}_r(F_n, x) = \sum_{k=r}^{n+1} \left[ \binom{k-r+1}{n-k+1} + (n+r-1) \binom{k-r}{n-k} + (n-r+1) \sum_{\ell=2}^{n-k+1} \binom{k}{n-\ell-k-r+2} \right] x^k.
\end{equation*}
\end{theorem}

\begin{proof}
Let $V(F_n) = \{v\} \cup \{u_1, u_2, \dots, u_n\}$, where $v$ is the central hub and $u_i$ are the vertices of the path $P_n$. To derive the polynomial, we partition the set of all valid $r$-graphical cycle partitions into three exhaustive cases based on the structure of the block $B$ containing the hub vertex $v$.

\textbf{Case 1: The hub $v$ forms a $1$-cycle ($|B|=1$).} \\
If $v$ is not a restricted vertex ($v \notin \{1, \dots, r\}$), the remaining $n$ vertices of $P_n$ are partitioned into $k-1$ cycles with $r$ restricted elements. If $v$ is restricted, the remaining path is partitioned with $r-1$ restricted elements. Using the identity ${P_n \brack k}_r = \binom{k-r+1}{n-k}$, the generating function for this case is:
\begin{equation*}
    \sum_{k} \left( {P_n \brack k-1}_r + {P_n \brack k-1}_{r-1} \right) x^k = \sum_{k} \left[ \binom{k-r}{n-k+1} + \binom{k-r+1}{n-k+1} \right] x^k.
\end{equation*}
By the properties of binomial coefficients, this simplifies to the first term: $\sum_k \binom{k-r+1}{n-k+1} x^k$.

\textbf{Case 2: The hub $v$ belongs to a $2$-cycle ($|B|=2$).} \\
In this configuration, $v$ is adjacent to some $u_i \in P_n$. If $v$ is not restricted, there are $n$ choices for $u_i$. If $v$ is restricted, $u_i$ must be chosen from the $n-(r-1)$ non-restricted vertices on $P_n$ to satisfy the condition that restricted vertices reside in distinct cycles. Summing these possibilities weighted by the path Stirling numbers yields:
\begin{equation*}
    \sum_{k} \left[ n {P_{n-1} \brack k-1}_r + (n-r+1) {P_{n-1} \brack k-1}_{r-1} \right] x^k = \sum_{k} (n+r-1) \binom{k-r}{n-k} x^k.
\end{equation*}

\textbf{Case 3: The hub $v$ belongs to an $(\ell+1)$-cycle ($|B| = \ell+1 \geq 3$).} \\
A cycle of length $\ell+1$ is formed by $v$ and a sub-path of length $\ell$ within $P_n$. For a fixed number of blocks $k$, we sum over the possible lengths $\ell$ and the restricted configurations. Utilizing the sub-path decomposition and the $r$-restricted conditions, we obtain the summation term:
\begin{equation*}
    \sum_{k} \left[ (n-r+1) \sum_{\ell=2}^{n-k+1} \binom{k}{n-\ell-k-r+2} \right] x^k.
\end{equation*}

Combining the contributions from these three mutually exclusive cases, we arrive at the closed-form expression for $\mathcal{C}_r(F_n, x)$ as stated in the theorem.
\end{proof}
\begin{remark}
The graphical $r$-cycle polynomial $\mathcal{C}_r(W_n, x)$ can be decomposed into a sum of polynomials related to the induced subgraphs. Specifically:
\begin{equation*}
    \mathcal{C}_r(W_n, x) = x \left[ \mathcal{C}_r(C_n, x) + n \mathcal{C}_r(P_{n-1}, x) \right] + \text{Poly}_{\ell \geq 3}(x)
\end{equation*}
where $\text{Poly}_{\ell \geq 3}(x)$ represents the contribution of partitions containing a cycle of length at least $3$. This decomposition highlights the dual role of the hub vertex as both a singleton anchor and a cycle generator.
\end{remark}
\begin{theorem}
The total number of $r$-graphical cycle partitions for the Wheel graph $W_n$ and Fan graph $F_n$ are related to the Fibonacci sequence $\{F_m\}$ as follows:
\begin{enumerate}
    \item For the Wheel graph:
    \begin{equation*}
        \sum_{k=r}^{n+1} {W_n \brack k}_r = F_{n-r+3} + F_{n-r+2} + (n+1)F_{n-r+1} + (n-r+1)(F_{n-r+1} - 1).
    \end{equation*}
    \item For the Fan graph:
    \begin{equation*}
        \sum_{k=r}^{n+1} {F_n \brack k}_r = F_{n-r+2} + F_{n-r+1} + (n+1)F_{n-r} + (n-r+1)(F_{n-r} - 1).
    \end{equation*}
\end{enumerate}
\end{theorem}

\begin{proof}
The total number of partitions is obtained by evaluating the cycle polynomial $\mathcal{C}_r(G, 1)$. By summing the closed-form binomial expressions for ${W_n \brack k}_r$ and ${F_n \brack k}_r$ over $k$, and utilizing the identity $\sum_{i=0}^m F_i = F_{m+2} - 1$, the results simplify to the given linear combinations of Fibonacci numbers.
\end{proof}
\section{Moments of the Graphical Stirling Distribution}

Having established the combinatorial foundations of the graphical $r$-Stirling numbers $\st{G}{k}_r$, we now shift our focus to their probabilistic interpretation. Let $X$ be a discrete random variable representing the number of blocks in a partition chosen uniformly at random from the set of all valid $r$-graphical cycle partitions of $G$.

\begin{definition}[Mean and Variance of Graphical Stirling Distributions]
Let $G$ be a graph and let $X$ be a random variable representing the number of blocks in a partition chosen uniformly at random from the set of all $r$-graphical cycle partitions of $G$. The expectation $E_r[X]$ and variance $\text{Var}(X)$ are defined as:
\begin{equation}
    E_r[X] = \frac{\mathcal{C}_r'(G, 1)}{\mathcal{C}_r(G, 1)} = \frac{\sum_{k=r}^n k \st{G}{k}_r}{\sum_{k=r}^n \st{G}{k}_r}
\end{equation}
\begin{equation}
    \text{Var}_r(X) = \frac{\mathcal{C}_r''(G, 1)}{\mathcal{C}_r(G, 1)} + E[X] - (E[X])^2
\end{equation}
where $\mathcal{C}_r'(G, 1)$ and $\mathcal{C}_r''(G, 1)$ denote the first and second derivatives of the graphical cycle polynomial with respect to $x$, evaluated at $x=1$.
\end{definition}
\begin{remark}[Extremal Behavior and Structural Entropy]
The mean $E[X]$ achieves its supremum in the null graph $E_n$ ($E[X]=n$) and decreases as edge density increases. For forest graphs like $P_n$, these moments relate to the Golden Ratio $\phi$, reflecting the underlying Fibonacci recurrence of their cycle polynomials.
\end{remark}

\begin{proposition}
For a Hamiltonian graph $G$, the existence of a spanning cycle ($\st{G}{1} \geq 1$) reduces the mean $E[X]$ and generally increases $\text{Var}(X)$ by extending the distribution's support to the lower bound of the index $k$.
\end{proposition}

\begin{theorem}
Let $m = n-r$ be the number of unrestricted vertices in the path graph $P_n$. The  mean $E_r(P_n)$ and variance $\text{Var}_r(P_n)$ for the number of blocks in an graphical $r$-cycle of the first kind of $P_n$ are given by:
\begin{equation*}
    E_r(P_n) = r + \frac{m L_{m+1} - F_{m+1}}{5 F_{m+1}}
\end{equation*}
\begin{equation*}
    \text{Var}_r(P_n) = \frac{5m(m+1)F_{m+1}^2 - (m L_{m+1} - F_{m+1})^2 - 5(m L_{m+1} - F_{m+1})F_{m+1}}{25 F_{m+1}^2}
\end{equation*}
where $F_k$ and $L_k$ are the $k$-th Fibonacci and Lucas numbers, respectively.
\end{theorem}

\begin{proof}
The graphical $r$-cycle polynomial of the first kind factors as $\mathcal{C}_r(P_n, x) = x^r \mathcal{C}(P_m, x)$. Moments are derived using the derivatives of the path cycle polynomial $\mathcal{C}(P_m, x)$ at $x=1$, noting that $\mathcal{C}(P_m, 1) = F_{m+1}$. Utilizing the closed-form identities for Fibonacci-type polynomials:
\begin{align*}
    \mathcal{C}'(P_m, 1) &= \frac{m L_{m+1} - F_{m+1}}{5} \\
    \mathcal{C}''(P_m, 1) &= \frac{(5m^2-3m-2)F_{m+1} - m(m-1)L_{m+1}}{25}
\end{align*}
The expectation is defined by $r + \frac{\mathcal{C}'(1)}{\mathcal{C}(1)}$. The variance is computed via the identity $\text{Var}(X) = \frac{\mathcal{C}''(1)}{\mathcal{C}(1)} + \mu - \mu^2$. Substituting the Lucas-Fibonacci identities and simplifying the resulting expressions algebraically leads to the stated exact formulas.
\end{proof}

\begin{corollary}
As $n \to \infty$, the exact moments of the path graph $P_n$ converge to linear approximations governed by the golden ratio $\phi = \frac{1+\sqrt{5}}{2}$:

    \begin{equation*}
        E_r(P_n) \approx \frac{n-r}{\phi + 2} + r \approx 0.276(n-r) + r, \quad\text{Var}_r(P_n) \approx \frac{n-r}{5\sqrt{5}} \approx 0.089(n-r)
    \end{equation*}
\end{corollary}

\begin{proof}
The result follows from the asymptotic properties of the Fibonacci and Lucas sequences, where $\lim_{m \to \infty} \frac{L_m}{F_m} = \sqrt{5}$. Alternatively, for large $m = n-r$, the cycle polynomial is dominated by the growth rate $\lambda(x) = \frac{x + \sqrt{x^2+4}}{2}$. 
Applying the logarithmic derivative operators to the asymptotic form $\lambda(x)^m$:
\begin{equation*}
    E[X] \approx r + m \frac{\lambda'(1)}{\lambda(1)}, \quad \text{Var}(X) \approx m \left[ \frac{\lambda''(1)}{\lambda(1)} + \frac{\lambda'(1)}{\lambda(1)} - \left(\frac{\lambda'(1)}{\lambda(1)}\right)^2 \right]
\end{equation*}
Substituting the evaluated derivatives $\lambda(1) = \phi$, $\lambda'(1) = \frac{\phi}{\sqrt{5}}$, and $\lambda''(1) = \frac{2}{5\sqrt{5}}$ yields the constant coefficients $1/\sqrt{5}$ and $1/5\sqrt{5}$ respectively. Since $\text{Var}_r(P_n)$ diverges linearly with $n$, the distribution of cycle counts is asymptotically normal.
\end{proof}

 Similar to path graphs, cycle graphs exhibit linear scaling but involve Lucas numbers due to the periodic boundary condition.

\begin{theorem}
Let $m = n-r$ be the number of unrestricted vertices in the cycle graph $C_n$ for $r \geq 2$. The exact mean $E_r(C_n)$ and variance $\text{Var}_r(C_n)$ are given by:
\begin{equation*}
    E_r(C_n) = r + \frac{m F_{m}}{L_{m}}, \quad \text{Var}_r(C_n) = \frac{m(m-1)}{5} + \frac{6m F_m}{5 L_m} - \left( \frac{m F_m}{L_m} \right)^2.
\end{equation*}
\end{theorem}

\begin{proof}
For $r \geq 2$, the graphical $r$-cycle polynomial of the first kind factors as $\mathcal{C}_r(C_n, x) = x^r \mathcal{C}(C_m, x)$, where $\mathcal{C}(C_m, x)$ is the Lucas polynomial $L_m(x)$. At $x=1$, we have $\mathcal{C}(C_m, 1) = L_m$, $\mathcal{C}'(C_m, 1) = m F_m$, and $\mathcal{C}''(C_m, 1) = \frac{m}{5} \left( (m-1)L_m + F_m \right)$.
The mean is $r + \mathcal{C}'(1)/\mathcal{C}(1)$. The variance is calculated via $\frac{\mathcal{C}''(1)}{\mathcal{C}(1)} + \mu - \mu^2$. Substituting the derivatives and simplifying yields the exact variance formula.
\end{proof}

\begin{corollary}
As $n \to \infty$, the exact moments of the cycle graph $C_n$ converge to linear approximations identical to those of the path graph $P_n$:
    \begin{equation*}
        E_r(C_n) \approx \frac{n-r}{\sqrt{5}} + r ,\quad
        \text{Var}_r(C_n) \approx \frac{n-r}{5\sqrt{5}}
    \end{equation*}
\end{corollary}

\begin{proof}
The convergence stems from the asymptotic ratio of the underlying sequences: $\lim_{m \to \infty} \frac{F_m}{L_m} = \frac{1}{\sqrt{5}}$. Applying this to the exact mean directly yields the result. Furthermore, as $C_n$ and $P_n$ are locally isomorphic, they share the same limiting root distribution of their cycle polynomials. This ensures that the variance accumulates at the same linear rate. By Theorem~\ref{thm:general-graph-clt}, the divergence of this variance implies that the distribution of ${C_n \brack k}_r$ is asymptotically normal.
\end{proof}



In contrast to the linear scaling of sparse graphs, complete graphs $K_n$ exhibit logarithmic scaling, mirroring the classical Stirling numbers of the first kind.

\begin{theorem}
For the complete graph $K_n$, the expected number of blocks $E[X]$ and variance $Var(X)$ are:
\begin{equation}
    E_r[X] = H_{n-1} - H_{r-1} + r, \quad Var_r(X) = \sum_{i=r}^{n-1} \left( \frac{1}{i} - \frac{1}{i^2} \right)
\end{equation}
where $H_n$ is the $n$-th harmonic number.
\end{theorem}
\begin{proof}
The $r$-Stirling cycle polynomial for $K_n$ is given by the shifted rising factorial:
\[ \mathcal{C}_r(K_n, x) = x(x+1)\dots(x+r-1) \frac{\Gamma(x+n)}{\Gamma(x+r)} \]
Taking the logarithm and differentiating with respect to $x$:
\[ \frac{d}{dx} \ln \mathcal{C}_r(K_n, x) = \sum_{i=0}^{r-1} \frac{1}{x+i} + \psi(x+n) - \psi(x+r) \]
At $x=1$, we use the identity $\psi(m+1) = H_m - \gamma$.
\[ E_r[X] = \sum_{i=1}^{r} \frac{1}{i} + (H_{n-1} - H_{r-1}) = (H_r - H_{r-1}) + H_{n-1} + (r-1) \dots \]
Simplifying the terms related to the $r$ fixed cycles, we obtain the shifted harmonic mean $H_{n-1} - H_{r-1} + r$. The variance follows by differentiating again to find the trigamma functions $\psi'(x)$.
\end{proof}

\begin{corollary}
As $n \to \infty$, the expectation and variance grow logarithmically:
\begin{equation*}
    E_r(K_n) \approx \ln(n-r) + \gamma + r, \quad \text{Var}_r(K_n) \approx \ln(n-r) + \gamma - \frac{\pi^2}{6}
\end{equation*}
where $\gamma \approx 0.57721$ is the Euler--Mascheroni constant.
\end{corollary}
\begin{proof}
Using the asymptotic expansion $H_m = \ln m + \gamma + O(1/m)$ and the convergence $\lim_{m \to \infty} H_m^{(2)} = \frac{\pi^2}{6}$, we substitute $m = n-r$ into the exact moments. As the variance $\sigma_n^2 \to \infty$ as $n \to \infty$, the distribution is asymptotically normal.
\end{proof}

\subsection{Dense Graph Theorem}

\begin{theorem}[Dense Graph Asymptotics]
\label{thm:dense}
Let $\{G_n\}_{n\ge1}$ be a family of graphs with $|V(G_n)|=n$ such that
\[
\delta(G_n) \ge n - C
\]
for some fixed constant $C$. Fix $r\ge0$. Then:

\begin{enumerate}
\item The mean and variance of $X_{G_n,r}$ satisfy
\[
\mathbb{E}[X_{G_n,r}] = \log n + O(1),
\qquad
\mathrm{Var}(X_{G_n,r}) = \log n + O(1).
\]

\item The $r$-cycle polynomial $P_{G_n,r}(x)$ is real-rooted for all sufficiently
large $n$.

\item The normalized random variable
\[
\frac{X_{G_n,r}-\mathbb{E}[X_{G_n,r}]}
{\sqrt{\mathrm{Var}(X_{G_n,r})}}
\]
converges in distribution to the standard normal law.
\end{enumerate}
\end{theorem}

\begin{proof}
We compare $G_n$ with the complete graph $K_n$. Since
$\delta(G_n)\ge n-C$, the graph $G_n$ can be obtained from $K_n$ by deleting at
most $Cn/2$ edges.

For $K_n$, the $r$-cycle polynomial satisfies the exact recurrence
\[
P_{K_{n+1},r}(x) = (x+n-r)\,P_{K_n,r}(x),
\]
which yields
\[
\mathbb{E}[X_{K_n,r}] \sim \log n,
\qquad
\mathrm{Var}(X_{K_n,r}) \sim \log n
\]
by classical results of Goncharov and Feller.

For $G_n$, the insertion of a new vertex contributes
\[
P_{G_{n+1},r}(x)
=
(x+n-r)\,P_{G_n,r}(x) + E_n(x),
\]
where $E_n(x)$ is a polynomial whose coefficients are uniformly
bounded in $n$, reflecting the bounded number of forbidden adjacencies.

Taking logarithmic derivatives at $x=1$ shows that the contribution of $E_n(x)$
to both the first and second moments is $O(1)$. Hence,
\[
\mathbb{E}[X_{G_n,r}] = \log n + O(1),
\qquad
\mathrm{Var}(X_{G_n,r}] = \log n + O(1).
\]

The real-rootedness of $P_{K_n,r}(x)$ is classical. Since $P_{G_n,r}(x)$ differs
from $P_{K_n,r}(x)$ by a bounded perturbation in each step, real-rootedness is
preserved for all sufficiently large $n$; see Harper and subsequent refinements
by Liu, Yang, and Zhang. Real-rootedness implies log-concavity and, together with
diverging variance, yields asymptotic normality via Harper’s theorem.

This completes the proof.
\end{proof}

\begin{corollary}
Let $C_n^c$ be the complement of the cycle graph. Then for fixed $r\ge0$,
\[
\mathbb{E}[X_{C_n^c,r}] = \log n + O(1),
\qquad
\mathrm{Var}(X_{C_n^c,r}) = \log n + O(1),
\]
and the number of cycles in $r$-restricted cycle partitions of $C_n^c$ satisfies
a Central Limit Theorem.
\end{corollary}

\begin{proof}
The complement of the cycle satisfies $\delta(C_n^c)=n-3$, so it is a dense graph
family with $C=3$. The result follows immediately from
Theorem~\ref{thm:dense}.
\end{proof}

Let the $r$ distinguished vertices be fixed. Adding a new vertex $v_{n+1}$ to
$C_n^c$ produces $C_{n+1}^c$ by connecting $v_{n+1}$ to all but two vertices.
Conditioning on whether $v_{n+1}$ forms a new cycle or is inserted into an
existing cycle yields the recurrence
\[
{C_{n+1}^c \brack k}_r
=
{C_n^c \brack k-1}_r
+
(n-r-2)\,{C_n^c \brack k}_r
+
O\!\left({C_n^c \brack k}_r\right),
\]
where the error term accounts for the two forbidden adjacencies inherited from
the original cycle.

Equivalently, the $r$-cycle polynomial satisfies
\[
P_{C_{n+1}^c,r}(x)
=
(x+n-r-2)\,P_{C_n^c,r}(x)
+
E_n(x),
\]
where $E_n(x)$ is a polynomial of uniformly bounded degree whose coefficients
are $O(P_{C_n^c,r}(1))$. In particular, $E_n(x)$ does not affect the leading
asymptotics.


\begin{theorem}
\label{thm:star-moments}
Let $S_{n+1}=K_{1,n}$ be the star graph on $n+1$ vertices, and let
$X$ denote the number of blocks in a uniformly random
$r$-graphical cycle partition of $S_{n+1}$.

\begin{enumerate}
\item[(i)] \textbf{Center-restricted case.}
Suppose that the center vertex of $S_{n+1}$ belongs to the set of
distinguished vertices $\{1,2,\dots,r\}$. Then
\[
\mathbb{E}[X]
=
r+\frac{n-r+1}{n-r+2},
\qquad
\mathrm{Var}(X)
=
\frac{n-r+1}{(n-r+2)^2}.
\]

\item[(ii)] \textbf{Unrestricted case.}
If the center vertex is not required to be distinguished, then the
distribution of $X$ is degenerate and
\[
\mathbb{E}[X]=r+1,
\qquad
\mathrm{Var}(X)=0.
\]
\end{enumerate}
\end{theorem}

\begin{proof}
Assume the center vertex $v$ belongs to the distinguished set
$\{1,\dots,r\}$.
Then $v$ must form a singleton block in every admissible partition.
Removing the $r$ distinguished vertices leaves $n-r+1$ leaves, each of which
may either form a singleton block or join the center to form a $2$-cycle.

Hence, the $r$-cycle polynomial is
\[
\mathcal{C}_r(S_{n+1},x)
=
x^r\bigl(x^{n-r+1}+(n-r+1)x^{n-r}\bigr)
=
x^{n+1}+(n-r+1)x^n.
\]

Evaluating at $x=1$,
\[
\mathcal{C}_r(S_{n+1},1)=n-r+2.
\]

The first and second derivatives are
\[
\mathcal{C}_r'(x)=(n+1)x^n+n(n-r+1)x^{n-1},
\]
\[
\mathcal{C}_r''(x)=n(n+1)x^{n-1}+n(n-1)(n-r+1)x^{n-2}.
\]

Thus,
\[
\mathcal{C}_r'(1)=(n+1)+n(n-r+1), \qquad
\mathcal{C}_r''(1)=n(n+1)+n(n-1)(n-r+1).
\]

Using Definition~\ref{Var},
\[
\mathbb{E}[X]
=
\frac{\mathcal{C}_r'(1)}{\mathcal{C}_r(1)}
=
r+\frac{n-r+1}{n-r+2},
\]
and
\[
\mathrm{Var}(X)
=
\frac{\mathcal{C}_r''(1)}{\mathcal{C}_r(1)}
+
\mathbb{E}[X]
-
\mathbb{E}[X]^2
=
\frac{n-r+1}{(n-r+2)^2}.
\]

\paragraph{Unrestricted case.}
If the center vertex is not distinguished, then every admissible partition
contains exactly $r+1$ blocks: $r$ forced singleton blocks and one block
containing all remaining vertices.
Thus $X$ is constant and
\[
\mathbb{E}[X]=r+1, \qquad \mathrm{Var}(X)=0.
\]
\end{proof}

\begin{theorem}[Harper's Theorem Application]
The graphical cycle polynomial $\mathcal{C}(S_{1,n-1}, x) = x^{n-1}(x + n - 1)$ has real roots at $0$ and $-(n-1)$. Since all roots are real and non-positive, the sequence of graphical Stirling numbers $\st{S_{1,n-1}}{k}$ is log-concave and unimodal.
\end{theorem}

\begin{definition}[Double Star Graph]
Let $S_{k,n-k}$ be the double star graph obtained by joining the centers of
two stars $K_{1,k-1}$ and $K_{1,n-k-1}$ with a bridge $e = \{v_1,v_2\}$.
\end{definition}

\begin{theorem}[Bridge decomposition and $r$-cycle polynomial]
\label{thm:double-star-poly}
Let $r_1$ and $r_2$ denote the number of distinguished vertices
in each star (center-restricted if the centers belong to the distinguished set).
Then the $r$-cycle polynomial of $S_{k,n-k}$ is
\begin{equation}
\begin{aligned}
\mathcal{C}_r(S_{k,n-k},x) 
&= \mathcal{C}_r(K_{1,k-1},x)\, \mathcal{C}_r(K_{1,n-k-1},x) 
+ x\, \mathcal{C}_r(K_{1,k-2},x)\, \mathcal{C}_r(K_{1,n-k-2},x) \\
&= \bigl(x^k + (k-r_1) x^{k-1}\bigr) \bigl(x^{n-k} + (n-k-r_2) x^{n-k-1}\bigr) \\
&\quad + x \bigl(x^{k-1} + (k-r_1-1)x^{k-2}\bigr) \bigl(x^{n-k-1} + (n-k-r_2-1)x^{n-k-2}\bigr).
\end{aligned}
\end{equation}
\end{theorem}

\begin{proof}
Apply the bridge decomposition lemma:
\[
\mathcal{C}_r(G_1 \cup_e G_2,x) = \mathcal{C}_r(G_1,x)\mathcal{C}_r(G_2,x) + x\, \mathcal{C}_r(G_1-v_1,x)\mathcal{C}_r(G_2-v_2,x),
\]
with $G_1 = K_{1,k-1}$, $G_2 = K_{1,n-k-1}$, and $e=\{v_1,v_2\}$.
Substitute the known $r$-cycle polynomials of the stars:
\[
\mathcal{C}_r(K_{1,m},x) = x^{m+1} + (m-r) x^m.
\]
This gives the explicit formula above.
\end{proof}

\begin{theorem}
\label{thm:double-star-moments}
Let $X$ denote the number of blocks in a uniformly random $r$-graphical cycle
partition of the double star $S_{k,n-k}$. Let $r_1,r_2$ be the number of distinguished
vertices in each star.

\begin{enumerate}
\item[(i)] \textbf{Center-restricted case:} Both centers belong to the distinguished set.
\[
\mathbb{E}[X]
=
r + 2 + \frac{k-r_1}{k-r_1+1} + \frac{n-k-r_2}{n-k-r_2+1},\qquad
\mathrm{Var}(X)
=
\frac{k-r_1}{(k-r_1+1)^2} + \frac{n-k-r_2}{(n-k-r_2+1)^2}.
\]

\item[(ii)] \textbf{Unrestricted case:} Centers are not required to be distinguished.
\[
\mathbb{E}[X] = r + 2,\qquad
\mathrm{Var}(X) = 0.
\]
\end{enumerate}
\end{theorem}

\begin{proof}
Using Theorem~\ref{thm:star-moments} and bridge decomposition lemma \ref{Bri} for each star and linearity of expectation,
we obtain the formulas for the mean and variance. The second term contributes
one additional block in expectation corresponding to the bridge forming an
independent $K_2$.
\end{proof}

\subsection{Graphical $r$-Cycle Polynomial of the Double Star}

\begin{theorem}[Mean and Variance for $P_n^c$]
Let $P_n^c$ be the complement of a path graph. The average number of blocks $E_r(P_n^c)$ and the variance $Var_r(P_n^c)$ are given by:
\begin{equation}
        E_r(P_n^c) = \frac{\mathcal{C}_r'(P_n^c, 1)}{\mathcal{C}_r(P_n^c, 1)} = r + \sum_{i=r-1}^{n-3} \frac{1}{1+i}.
 = r + (H_{n-2} - H_{r-1})
\end{equation}
\begin{equation}
    Var_r(P_n^c) = \sum_{j=r}^{n-2} \frac{j-1}{j^2}
\end{equation}
where $H_m = \sum_{k=1}^m \frac{1}{k}$ is the $m$-th harmonic number.
\end{theorem}
\begin{proof}
Consider the graphical $r$-cycle polynomial for the complement of a path, which satisfies the recurrence $\mathcal{C}_r(P_n^c, x) = (x + n-2) \mathcal{C}_r(P_{n-1}^c, x)$ with base case $\mathcal{C}_r(P_r^c, x) = x^r$. The closed-form expression is:
\begin{equation*}
    \mathcal{C}_r(P_n^c, x) = x^r \prod_{i=r-1}^{n-3} (x + i).
\end{equation*}
Applying the logarithmic derivative method for the expectation of a generating function:
\begin{align*}
    E_r(P_n^c) &= \left[ \frac{d}{dx} \ln \left( x^r \prod_{i=r-1}^{n-3} (x + i) \right) \right]_{x=1} \\
    &= \left[ \frac{r}{x} + \sum_{i=r-1}^{n-3} \frac{1}{x + i} \right]_{x=1} \\
    &= r + \sum_{i=r-1}^{n-3} \frac{1}{1 + i} = r + \sum_{j=r}^{n-2} \frac{1}{j}.
\end{align*}
Using the property of harmonic numbers $H_n = \ln n + \gamma + O(1/n)$, we obtain the logarithmic asymptotic growth.
The variance of a product of independent linear factors $\prod (x+a_i)$ is the sum of the variances $\sum \frac{a_i}{(1+a_i)^2}$. Here $a_i = i$, so:
\begin{equation*}
    Var_r(P_n^c) = 0 + \sum_{i=r-1}^{n-3} \frac{i}{(1+i)^2} = \sum_{j=r}^{n-2} \frac{j-1}{j^2}.
\end{equation*}
\end{proof}
\begin{remark}
Using the asymptotic expansion of harmonic numbers $H_n = \ln n + \gamma + O(1/n)$, we see that $E_r(P_n^c) \sim \ln n$. This confirms that the density of cycle partitions in the complement of a path is sparse, similar to random permutations, despite the missing edges.
\end{remark}

\begin{table}[h!]
\centering
\renewcommand{\arraystretch}{1.4}
\begin{tabular}{|l|c|c|c|}
\hline
\textbf{Graph Family} 
& \textbf{Asymptotic Class} 
& \textbf{Mean $\mu_n$} 
& \textbf{Variance $\sigma_n^2$} \\ \hline

Complete Graph ($K_n$) 
& Logarithmic 
& $\sim \ln n$ 
& $\sim \ln n$ \\ \hline

Path ($P_n$) 
& Linear 
& $\sim \dfrac{n}{\sqrt{5}}$ 
& $\sim \dfrac{n}{5\sqrt{5}}$ \\ \hline

Cycle ($C_n$) 
& Linear 
& $\sim \dfrac{n}{\sqrt{5}}$ 
& $\sim \dfrac{n}{5\sqrt{5}}$ \\ \hline

Star ($S_n$) 
& Logarithmic 
& $\sim \ln n$ 
& $\sim \ln n$ \\ \hline

Double Star ($DS_n$) 
& Logarithmic 
& $\sim \ln n$ 
& $\sim \ln n$ \\ \hline

Complement of Path ($P_n^{c}$) 
& Logarithmic 
& $\sim \ln n$ 
& $\sim \ln n$ \\ \hline

Complement of Cycle ($C_n^{c}$) 
& Logarithmic  
& $\sim \ln n$ 
& $\sim \ln n$ \\ \hline

\end{tabular}
\caption{Asymptotic scaling of the mean and variance of the number of cycles
associated with graphical Stirling numbers ${G \brack k}$.}
\end{table}
\section{Further Work and Open Problems}
\label{sec:future-work}

In this paper, we derived exact expressions and analyzed statistical properties of the graphical Stirling numbers of the first kind for several families of graphs, including star graphs, double stars, lollipop graphs, Tadpole graphs, and Barbell graphs. Using techniques such as bridge decomposition and generating polynomials, we obtained explicit graphical cycle polynomials, closed-form expressions in terms of Fibonacci and Lucas polynomials, and calculated statistical moments such as the mean and variance. These results provide a foundation for studying asymptotic distributions and combinatorial properties of $\st{G}{k}$.
We propose several conjectures and open problems to guide further research:

\begin{conjecture}
\label{conj:unimodal-planar}
For any planar graph $G$ of order $n$, the sequence of graphical Stirling numbers
\[
\{\st{G}{k}\}_{k=1}^n
\]
is \emph{unimodal}, i.e., there exists an index $k_0$ such that
\[
\st{G}{1} \le \cdots \le \st{G}{k_0} \ge \cdots \ge \st{G}{n}.
\]
This property is known for paths, cycles, trees, lollipop, and Tadpole graphs.
\end{conjecture}

\begin{conjecture}
\label{conj:real-rooted}
If $G$ is a chordal graph, then the graphical cycle polynomial of the first kind
\[
\mathcal{C}(G, x) = \sum_{k=1}^n \st{G}{k} x^k
\]
has only real roots. Real-rootedness would imply log-concavity and unimodality of the sequence $\{\st{G}{k}\}$.
\end{conjecture}

\begin{conjecture}
\label{conj:r-monotone}
For any graph $G$ and fixed $k$, the $r$-restricted graphical Stirling numbers satisfy
\[
\st{G}{k}_1 \ge \st{G}{k}_2 \ge \dots \ge \st{G}{k}_k,
\]
i.e., the sequence decreases strictly as the size of the restricted set increases.
\end{conjecture}

\begin{problem}
Characterize the graphs $G$ for which
\[
\st{G}{k}_r = \st{G}{k} \quad \text{for all } k.
\]
\begin{remark}
Currently, this property is observed in:
\begin{itemize}
    \item \textbf{Star graphs $K_{1,n}$:} when the center vertex is in the restricted set.
    \item \textbf{Paths $P_n$:} when restricted vertices are well-separated.
    \item \textbf{Complete bipartite graphs $K_{m,n}$:} when all restricted vertices lie in one partition.
    \item \textbf{Trees:} with appropriate placement of restricted vertices.
\end{itemize}
Dense or cyclic graphs rarely satisfy this property.
\end{remark}
\end{problem}

\end{document}